\documentclass{article}

\usepackage{longtable}
\usepackage{amsmath}
\usepackage{amssymb, mathabx, amscd, amsthm, amsfonts}
\usepackage[pdftex]{graphicx}

\usepackage[most]{tcolorbox}
\tcbuselibrary{breakable, skins}
\usepackage{etoolbox}

\usepackage{algcompatible}
\usepackage{enumitem}
\usepackage{tikz}
\usetikzlibrary{cd}
\usepackage{geometry}
\usepackage{xcolor}
\usepackage{xcolor}
\usepackage{tabularray}
\usepackage{algorithm}
\usepackage{etoolbox}
\usepackage{algpseudocode}
\usepackage{fixltx2e}
 \usepackage{indentfirst}
\usepackage{pgf-pie}

\definecolor{amethyst}{rgb}{0.6, 0.4,0.8}
\definecolor{indigo}{rgb}{.2109, .1016, .8359}

\newenvironment{staticalgorithm}{%
  \refstepcounter{algorithm}
  \par\addvspace{\abovecaptionskip}%
  \hrule height 0.8pt 
  \vspace{3pt}
  \noindent%
}{%
  \hrule height 0.8pt 
  \addvspace{\belowcaptionskip}
  \vspace{6pt}
}

\newcommand{\staticalgorithmcaption}[3]{
  \hspace{-10pt} \textbf{Algorithm \thealgorithm:} #1\\ 
  \vspace{-6pt} 
  \hrule height 0.4pt 
  \vspace{5pt}
  \noindent
  \textbf{Input:} #2 \\
  \textbf{Output:} #3 \\
  \vspace{-6pt}
  \hrule height 0.4pt 
  \vspace{-3pt}
  \noindent
}

\usepackage[backend=biber, style=alphabetic, backref, maxnames=99]{biblatex}
\usepackage{hyperref} 

\usepackage{caption}
\hypersetup{
    colorlinks=true,
    linkcolor=blue,
    filecolor=magenta,      
    urlcolor=cyan}

\newcommand{\Fl}{\mathbb{F}_\ell}
\newcommand{\Fp}{\mathbb{F}_p}
\newcommand{\FF}{\mathbb{F}}
\newcommand{\Q}{\mathbb{Q}}
\newcommand{\Z}{\mathbb{Z}}
\newcommand{\R}{\mathbb{R}}
\newcommand{\C}{\mathbb{C}}
\renewcommand{\O}{\mathcal{O}}
\newcommand{\Spec}{{\rm Spec}}
\newcommand{\Aut}{{\rm Aut}}
\newcommand{\End}{{\rm End}}

\newcommand{\Gal}{{\rm Gal}}
\newcommand{\GL}{{\rm GL}}
\newcommand{\GSp}{{\rm GSp}}
\newcommand{\Frob}{{\rm Frob}}
\newcommand{\im}{\operatorname{im}}

\newcommand{\Jac}{{\rm Jac}}

\newcommand{\p}{\mathfrak{p}}

\newcommand{\PP}{\mathbb{P}}

\theoremstyle{definition}
\newtheorem{defi}{Definition}[subsection]

\newtheorem{rmk}[defi]{Remark}
\newtheorem{cor}[defi]{Corollary}
\newtheorem{thm}[defi]{Theorem}
\newtheorem{lemma}[defi]{Lemma}

\newtheorem*{introthmm}{Theorem}

\usepackage{color}

\title{Mod-5 Galois images from abelian surfaces}
\author{Aidan Hennessey, Mathilde Kermorgant, and Andy Zhu}
\date{}

\begin{document}

\maketitle

\begin{abstract}
    For $J$ an abelian surface, the Galois representation $\rho_{J, \ell} : \Gal(\overline{\Q}/\Q) \rightarrow \Aut(J[\ell]) \simeq \GSp_4(\Fl)$ is typically surjective, with smaller images indicating extra arithmetic structure. It is already known how to probabilistically compute whether $\rho_{J, \ell}$ is surjective, and recent work by Chidambaram computes $\im \rho_{J, \ell}$ for $\ell = 2, 3$. We probabilistically compute $\rho_{J, 5}$ for the Jacobians of 95\% of genus 2 curves in the L-functions and Modular Forms Database (LMFDB) for which $\rho_{J, 5}$ is not yet known. For the remaining Jacobians, we determine the order of the image and give a short list of candidate images.
\end{abstract}

\section*{Introduction}

Let $C/\Q$ be a genus $g$ hyperelliptic curve. Let $J = \Jac(C)$ be the Jacobian of $C$. There is a natural action of the absolute Galois group $\Gal(\overline{\Q}/\Q)$ on $C$ which extends to an action on $J$.

Let $J[n]$ denote the $n$-torsion subgroup of $J$. As a group, $J[n]$ is isomorphic to $(\Z/n\Z)^{2g}$. The multiplication-by-$n$ map on $J$ is given by rational equations, so $J[n]$ is a subvariety defined over $\Q$. Because $J[n]$ is a subvariety, the above-defined action of $\Gal(\overline{\Q}/\Q)$ restricts to an action of $\Gal(\overline{\Q}/\Q)$ on $J[n]$. Since the group law of the Jacobian is defined by rational equations, the absolute Galois group respects it, acting by group automorphisms. This gives a map $\rho_{J, n}: \Gal(\overline{\Q}/\Q) \rightarrow \Aut_{\mathsf{Gp}}(\Z/n\Z)^{2g}$. For $n=\ell$ a prime, we identify $\Aut_{\mathsf{Gp}}(\Z/n\Z)^{2g} \cong \GL_{2g}(\Fl)$. 

\begin{defi}[Mod-$\ell$ Galois representation from torsion]
    For a Jacobian $J$ of a genus $g$ hyperelliptic curve and a prime $\ell$, the \emph{Galois representation from $\ell$-torsion on $J$} is the map $\rho_{J, \ell}: \Gal(\overline{\Q}/\Q) \rightarrow \GL_{2g}(\Fl)$.
\end{defi}

The image of the mod-$\ell$ Galois representation is of interest because it encodes various arithmetic information. For example, for $J$ the Jacobian of a genus 2 curve, $J$ admits a rational $\ell$-torsion point if and only if $\im \rho_{J, \ell}$ is conjugate to a subgroup of $G_1$, and $J$ admits an $\ell$-isogeny if and only if $\im \rho_{J, \ell}$ is conjugate to a subgroup of $G_2$, where
\[
    G_1 = \begin{bmatrix}
        1 &*&*&*\\
        0&*&*&*\\
        0&*&*&*\\
        0&*&*&*\\
    \end{bmatrix}
    \quad \text{ and } \quad
    G_2 = \begin{bmatrix}
        *&*&*&*\\
        0&*&*&*\\
        0&*&*&*\\
        0&*&*&*\\
    \end{bmatrix}.
\]

It is difficult to rigorously determine the image of Galois, but there is a large body of work using probabilistic methods to determine high-likelihood candidate images. Sutherland \cite{Elliptic-curve-images} computed $\im \rho_{E, \ell}$ using probabilistic methods for all elliptic curves $E$ without complex multiplication in the LMFDB and all primes $\ell$. For genus 2 curves, it is a theorem of Serre \cite[pp.~50--51]{serre-open-image} that if $\End\, J(\overline{\Q}) =\Z$ then $\rho_{J, \ell}$ surjects onto $\GSp_4(\Fl)$ for all but finitely many primes. For each such curve in the LMFDB, Banwait, Brumer, Kim, Klagsbrun, Mayle, Srinivasan, and Vogt \cite{Vogt-surjectivity} developed a probabilistic algorithm to compute the finite set of non-surjective primes. Some work towards computing the precise image in non-surjective cases already exists. In the $\ell = 2$ case, it suffices to consider the Weierstrass points to identify $\Q(J[2])$. The $\ell = 3$ case is addressed in recent work by Chidambaram \cite{mod-3-image-paper}.

In this paper, we address the $\ell = 5$ case. Our main contribution lies in developing an algorithm to compute the likely image of $\rho_{J,5}$ for a Jacobian $J$. 

\begin{introthmm}[Main Theorem] \label{thm: correct-imager}
    Let $C/\Q$ be a genus 2 hyperelliptic curve with Jacobian $J$. There exists an effective constant $N$, depending only on $C$, such that there exists an algorithm sampling all primes in the range $[10000, N]$ 
    that produces a list of at most eight equal-order subgroups containing the mod-$5$ image of Galois.
\end{introthmm}

The parameter $N$ in the Main Theorem is a computationally intractable bound. Hence, we compute likely images by sampling a smaller set of primes. 

\begin{introthmm}[Computational Results]
    Let $J$ be the Jacobian of one of the 3990 genus two curves in the \cite{lmfdb} for which $\rho_{J,5}$ is not known to be surjective. On input $J$, Algorithm \ref{alg: imager} produces a set of at most five equal-order likely images of $\rho_{J,5}$. Furthermore, for 3867 of the curves described above, Algorithm 1 produces a single likely image.
    
    Assume an independent random distribution of Frobenius elements and that, as a prior, every subgroup of $\GSp_4(\FF_5)$ with surjective similitude character is an equally likely candidate image. Then, the probability that $\rho_{J, 5}$ is \emph{not} on the list of subgroups output by Algorithm \ref{alg: imager} is bounded above by 1 in $9.5$ billion.
\label{thm: computational results}
\end{introthmm}

In Section \ref{sec: background} we review key background about the structure of Jacobians of genus 2 curves and their $5$-torsion subgroups. We also discuss Frobenius elements and restrictions on the possible images of Galois. In Section \ref{sec: Algorithm}, we describe Algorithm \ref{alg: imager} in detail and provide a proof of the Main Theorem.
In Section \ref{sec: results}, we prove the error bound claimed in the Computational Results Theorem, present the result  of running Algorithm \ref{alg: imager} on the genus 2 curves in the LMFDB, and give some additional information about those images.

Code for this work can be found in our Github repository
\begin{center}
    \href{https://github.com/maathilde-k/Mod-5-Galois-Images-of-Genus-2-Abelian-Curves/}{\texttt{https://github.com/maathilde-k/Mod-5-Galois-Images-of-Genus-2-Abelian-Curves/}}.
\end{center}

Throughout the paper, let $f \in \Z[x]$ be a a polynomial of degree 5 or 6, $C$ be the corresponding genus 2 hyperelliptic curve given by $y^2 = f(x)$, and $J = \Jac(C)$ be the Jacobian of $C$. Let $\im \rho_{J, 5}$ denote the image of the associated mod-5 Galois representation, and refer to this as the ``image of Galois". Generally, one has the $\overline{\Q}$-endomorphism group $\End(J) \cong \Z$ with the endomorphisms being multiplication-by-$n$ maps. We refer to curves $C$ with $\End(\Jac(C_{\overline{\Q}})) = \Z$ as \emph{typical}. We refer to curves whose Jacobians admit additional endomorphisms as \emph{atypical}.

Throughout our paper, we label subgroups of $\GSp_4(\FF_\ell)$ according to the labeling scheme used by \cite{lmfdb}. These labels take the format $[\ell.i.j]$, where $\ell$ is the characteristic of the ground field $\Fl$. The second parameter, $i$, indicates the index of the subgroup in $\GSp_4(\Fl)$. The final parameter, $j$, is simply a position in the list of conjugacy classes of index $i$ subgroups.

\subsection*{Acknowledgements}

We thank Sachi Hashimoto and Isabel Vogt for their mentorship, without which this project would not have been possible. We are also grateful to Edgar Costa and MIT for technical support and access to the Châtelet computing cluster, and to Andrew Sutherland and Kiran Kedlaya for access to \texttt{SmallJac}, with additional thanks to Andrew Sutherland for \texttt{GSPLattice}. We thank Shiva Chidambaram for his helpful answers to our questions. This project was funded by Brown University and supervised by Sachi Hashimoto and Isabel Vogt.

\section{Background} \label{sec: background}

\subsection{The possible images of Galois}

\begin{defi}[General symplectic group, $\GSp_{2g}(F)$, similitude character] Let $F$ be a field, and $A$ a non-degenerate symplectic form over $F$. Up to conjugation, we may take $A$ to be 
\[ A = \begin{pmatrix} 0 & I_n \\ -I_n & 0\end{pmatrix}.\]
The \textit{general symplectic group} $\GSp_{2g}(F)$ is the subgroup $$\GSp_{2g}(F) := \left\{M \in \GL_{2g}(F) : M^T A M = \lambda A, \quad \lambda \in F^{\times} \right\}$$ of $\GL_{2n}(\mathbb{F}_{\ell})$ which preserves $A$ up to scalars. In the notation of the above line, the map $M \mapsto \lambda$ is a group homomorphism $\GSp_{2g}(F) \rightarrow F^\times$, and is known as the \emph{similitude character}.
\end{defi}

The Weil pairing furnishes a non-degenerate symplectic form on the $5$-torsion, and is preserved up to scalars by the action of Galois. As such, the codomain of $\rho_{J, 5}$ may be restricted to $\GSp_4(\FF_{5})$. Furthermore, over $\Q$, the image of the Galois representation must have surjective similitude character, 
restricting the possibilities for $\im \rho_{J, 5}$. Additionally, complex conjugation acts on $J[5]$ as a non-trivial involution and has similitude $-1$ (since it sends $\zeta_5 \mapsto \zeta_5^{-1}$). Thus, the image of Galois must include an element of order 2 and similitude -1. This leaves (up to conjugacy) 1125 subgroups of $\GSp_4(\FF_5)$ for consideration. We therefore work on distinguishing between these 1125 subgroups.

One way subgroups may be distinguished is by invariant subspaces. The following notion will be helpful.
\begin{defi}[$(\Lambda)$-group eigenspace]
    \label{def: group_eigenspace}
    Let $\Lambda \subseteq F^\times$ be a subset of the unit group of a field $F$. Let $V$ be a vector space over $F$ and let $G \leq \Aut(V)$ be a subgroup of the automorphism group of $V$. A $\Lambda$-group eigenspace of $G$ is a subspace $W \subseteq V$ such that for all automorphisms $g \in G$, $W$ is an eigenspace of $g$ with eigenvalue $\lambda_g \in \Lambda$.
\end{defi}

\label{sec: possible_images}

\subsection{Mumford Coordinates}

Throughout the paper, we denote points of the Jacobian by their \emph{Mumford coordinates}. 

\begin{defi} [Mumford coordinates]
    Let $C$ be a genus 2 curve with two points at infinity (rather than a Weierstrass point). Let $D$ be a divisor class on $C$, with reduced divisor $[P_1] + [P_2] - ([\infty_1]+[\infty_2])$, where $P_1$, $P_2$ are points on $C$ and $\infty_1$ and $\infty_2$ are the points at infinity. The \emph{Mumford coordinates} $(u,v)$ of $D$ are the unique pair of polynomials $u$ and $v$ such that:
    \begin{enumerate}
        \item $u$ is monic of degree at most $2$,
        \item $u(x(P_1)) = 0$ and $u(x(P_2)) = 0$,
        \item $v$ is of degree at most $1$,
        \item $u \,\vert\, v^2 - f$.
    \end{enumerate}
\end{defi}

\begin{rmk}
    If $P_1$ or $P_2$ lies at infinity, $u$ has degree less than 2. If $P=Q$, $u$ has a double root at their $x$ coordinate.
\end{rmk}

\begin{rmk} 
     A point on the Jacobian is defined over $K$ if and only if all the coefficients of the polynomials comprising its Mumford coordinates lie in $K$.
\end{rmk}

\subsection{Frobenius Elements}

The most direct approach to computing the image of a Galois representation would be to compute the $\ell$-torsion points over $\overline{\Q}$ and their images under different field automorphisms. However, the 5-torsion field of an abelian surface is too large to compute explicitly for a large number of curves. Taking inspiration from \cite{Elliptic-curve-images} and \cite{Vogt-surjectivity}, we instead collect information by sampling Frobenius elements.

\begin{defi}[5-torsion field]
    Let $C$ be a genus 2 curve defined over $\Q$ with Jacobian $J=\Jac(C)$. Let $\Q(J[5])$ denote the \emph{5-torsion field of $J$}, defined to be the smallest field over which all the 5-torsion points are defined. Concretely, the 5-torsion field may be obtained by adjoining to $\Q$ the coefficients of the Mumford coordinates of the 5-torsion points.
    \label{def: 5-torsion_field}
\end{defi}

\begin{rmk}
    The 5-torsion field is a Galois extension.
\end{rmk}

Let $p$ be a prime of good reduction. The action of the Frobenius element $\Frob_p$ on the 5-torsion of an abelian variety $J$ is the same as the action of the Frobenius automorphism of $\overline{\FF}_p$ on the 5-torsion $J(\overline{\FF}_p)[5]$ over $\Fp$. Thus, properties of Frobenius elements may be calculated indirectly from structure of the reduction of $J$ modulo $p$. We refer to such calculations as ``local methods". Frobenius elements are representative of the Galois group in the sense that their images under $\rho_{J,5}$ equidistribute in conjugacy classes, as formalized in the following theorem.

\begin{thm}[Chebotarev Density Theorem]
    Let $K/\Q$ be a Galois extension. Let $\mathcal{C}$ be a conjugacy class of $G = \Gal(K/\Q)$. Let $$P_\mathcal{C}(n) = \#\{\text{prime } p \>|\> \Frob_p \in \mathcal{C}, p < n\}$$ denote the number of primes $p$ less than $n$ such that $\Frob_p \in \mathcal{C}$. Let $\pi(n)$ denote the number of primes less than $n$. Then, we have that $$\lim_{n \rightarrow \infty} \frac{P_{\mathcal{C}}(n)}{\pi(n)} = \frac{\#\mathcal{C}}{\#G}.$$
    \label{thm: chebotarev}
\end{thm}

By applying the Chebotarev density theorem to $K=\Q(J[5])$, we may use statistical methods to infer the likely image of the Galois representation. Effective versions of the Chebotarev density theorem, such as that in \cite{effective}, give bounds on the error of such statistical methods. Thus, adequately large samples could give results not just with arbitrarily small probability of error, but with 0 probability of error. However, even assuming the generalized Riemann Hypothesis, the required sample sizes are intractably large.

We compute the characteristic polynomial of the Frobenius element $\Frob_p$ via the $L$-polynomial of the reduction of $J$ mod $p$. By the Weil Conjectures, the L-polynomial of $J(\Fp)$ is the reciprocal polynomial of the characteristic polynomial of the Frobenius element $\Frob_p$. Computing L-polynomials can be done efficiently using existing software such as Kedlaya and Sutherland's \cite{smalljac} \texttt{Smalljac} program. 

Another piece of local information we extract from $\Frob_p$ is the dimension of its 1-eigenspace as an automorphism of $J[5] \cong \FF_5^4$. Because only the points defined over $\Fp$ are fixed by the Frobenius automorphism, one may compute the dimension of the $1$-eigenspace of $\Frob_p$ by  counting the number of $5$-torsion points over $\Fp$. 

\begin{rmk}
    Implicitly, the distribution of these dimensions over many sampled primes (the ``1-eigenspace spectrum"), encapsulates information about the overall order of the group. A martix is the identity if and only if it has a 4-dimensional 1-eigenspace. Thus, one may estimate the order of the group $\im \rho_{J, 5}$ by taking the reciprocal of the proportion of sampled Frobenius elements with a 4-dimensional 1-eigenspace.
    \label{rmk: order}
\end{rmk}

\begin{lemma}[To be used in proof of \ref{lem: distribution-filter}]
    Among all elements of $\GSp_4(\FF_5)$, there are exactly 99 distinct (characteristic polynomial, 1-eigenspace dimension) pairs.
    \label{lem: 99-vector space}
\end{lemma}

\begin{proof}
This is a straightforward computation; see the \href{https://github.com/maathilde-k/Mod-5-Galois-Images-of-Genus-2-Abelian-Curves/blob/main/Paper%20Lemmas/99pairs.magma}{Paper Lemmas} folder in the GitHub repository.
\end{proof}

\subsection{Height Bounds for Torsion Points}

\label{sec: height_bounds}

A main technique of this paper is computation of 5-torsion subgroups over quadratic fields via direct search for rational points. Heights on Jacobians and bounds on the height of torsion are important tools in constraining the points over which to search. This section reviews the relevant definitions and states a height bound theorem used in our search.

\begin{defi}[(Logarithmic) heights in $\PP^n(K)$]
    Let $K$ be a number field, and let $P = [x_0 : ... : x_n] \in \PP^n(K)$.
    Let $M_K$ denote the set of places of $K$, and let $\{|\cdot|_v : v \in M_K\}$ denote a set of absolute values on $K$, normalized so that a product formula holds. Then, the \emph{(logarithmic) height} $h_{\PP^n(K)}(P)$ of $P$ is given by $$h_{\PP^n(K)}(P) = \frac{1}{[K:\Q]}\log\left(\prod_{v \in M_K} \max\{|x_0|_v, ..., |x_n|_v\}\right).$$
    \label{def: projective_height}
\end{defi}

\begin{rmk}
    The height $h_{\PP^n(K)}(P)$ as defined above is independent of choice of representative of $P$ precisely because the absolute values are normalized towards a product formula. The quotient by $[K:\Q]$ in $h_{\PP^n(K)}$ is not essential in our context, but is nice in that it extends to a height on $\PP^n(\overline{\Q})$.
\end{rmk}

\begin{defi}[Kummer surface associated to a Jacobian]
    The Jacobian admits a negation involution $P \mapsto -P$. The quotient of $J$ by this involution is ramified precisely at the 2-torsion. The image of the quotient map is the \emph{associated Kummer surface}.
    \label{def: kummer_surface}
\end{defi}

\begin{defi}[Standard embedding of the Kummer surface]
    Let $J$ be the Jacobian of the genus 2 curve $C : y^2 = f(x)$, with the $x^n$ coefficient of $f$ denoted by $a_n$. Let $\phi: J \rightarrow \PP^3$ be the unique map which, on the dense open set where the following makes sense, is given by $$\phi: P = (x^2 + \alpha x + \beta, \delta x + \gamma) \mapsto [1 : -\alpha : \beta : m],$$ where $$m = \frac{2a_6\beta^3 - a_5\alpha\beta^2 + 2a_4 \beta^2 - a_3\alpha\beta + a_2\beta - a_1\alpha + 2a_0 - 2(\gamma^2 \beta - \gamma \delta \alpha + \delta^2)}{\alpha^2 - 4\beta}.$$  This map, first appearing in \cite[pp.~6--19]{middlebrow}, is the \emph{standard embedding} of the Kummer Surface associated to $J$.
    \label{def: standard_embedding}
\end{defi}

\begin{defi}[Na\"ive height] 
    Let $J$ be the Jacobian of a genus 2 hyperelliptic curve over $\Q$, and let $\phi$ be the standard embedding of its associated Kummer surface. Then, the naive height $h(P)$ of a point $P \in J(K)$ is given by $h(P) = h_{\PP^3(K)}(\phi(P))$.
    \label{def: naive_height}
\end{defi}

\begin{rmk}
    Note that there are other possible embeddings of the Kummer surface associated to $J$ into $\PP^3$, and the naive height is dependent on the particular embedding chosen. The explicit embedding in Definition \ref{def: standard_embedding} is common in the literature; see \cites[pp.~184--185]{height-bounds}[p.~334]{canonical-heights}. 
    \label{rmk:choice-of-embedding}
\end{rmk}

\begin{lemma}
    Suppose $\alpha = a/b$ and $\beta = c/d$ are rational numbers expressed in lowest terms with $b$ and $d$ positive. Then, the na\"ive height of the point $(x^2 + \alpha x + \beta, \gamma x + \delta) \in J(\Q)$ is lower-bounded by $\max(\log(b), \log(d))$.
\end{lemma}

\begin{proof}

    Observe that if one starts with a point $[x_0:...:x_n] \in \PP^n$ and deletes some number of coordinates, the place-wise contributions to the height of the resulting point $[x_{i_0}:x_{i_1}:...:x_{i_{n'}}] \in \PP^{n'}$ all weakly decrease. Thus, the height weakly decreases as well.

    The na\"ive height of $P$ is thus lower-bounded by both $h_{\PP^1(\Q)}([1:-\alpha])$ and $h_{\PP^1(\Q)}([1:\beta])$. Writing $\alpha$ and $\beta$ as in the lemma statement, these heights are in turn lower bounded by $\log b$ and $\log d$. Chaining these bounds together gives the lemma statement.
\end{proof}

\begin{thm}[Corollary 8.1 in \cite{height-bounds}]
    Let $J$ be the Jacobian of the genus 2 curve given by $C:y^2=f(x)$ with $f \in \Q[x]$. Fix $d \in \Z_+$. There is an effectively computable upper bound $B$, dependent only on $f$ and $d$, on the na\"ive height of any torsion point of $J$ defined over a degree $d$ number field. The cited paper includes an explicit formula for this bound.
    \label{thm: height-bound}
\end{thm}

\section{The Algorithm} \label{sec: Algorithm}

Given a hyperelliptic curve $C$, we compute two types of information about $C$ --- local information and global information.

Local information is data obtained from sampling Frobenius elements, namely their characteristic polynomials and $1$-eigenspace dimensions. Global information consists of data obtained by examining $J = \Jac(C)$ over number fields. These data are used to narrow down the possible images of $\rho_{J,5}$, as laid out in Algorithm \ref{alg: imager} below. The algorithm is implemented in Magma \cite{Magma} and is available at the following repository:
\begin{center}
\href{https://github.com/maathilde-k/Mod-5-Galois-Images-of-Genus-2-Abelian-Curves/tree/main}{\texttt{https://github.com/maathilde-k/Mod-5-Galois-Images-of-Genus-2-Abelian-Curves/tree/main}}.
\end{center}
To illustrate the algorithms described, we employ the curve $y^2=4x^5-20x^3+5x^2+20x-4$ (LMFDB label \href{https://www.lmfdb.org/Genus2Curve/Q/431250/a/431250/1}{431250.a.431250.1}) as a running example.

\begin{staticalgorithm}
\staticalgorithmcaption{An algorithm to compute the image of the mod-5 Galois representation.}{A genus $2$ curve $C$, and a prime bound $N$, and a confidence threshold $\nu$.}{A short list of subgroups of $\GSp_4(\FF_5)$, which likely (and for large enough $N$, provably) contains the image of the mod-5 Galois representation associated to the Jacobian of $C$.} 
\label{alg: imager}
\begin{enumerate}
    \item Compute the subgroup lattice of $\GSp_4(\FF_5)$ and initialize \texttt{possibilities} as the set of all (conjugacy classes of) subgroups with surjective similitude character and order $2$ element with similitude character $-1$ corresponding to complex conjugation.

    \item Set $U = 20000$.

    \item \textbf{Do}
    \begin{enumerate}
        \item Using Algorithm \ref{alg: frob_sample}, compute the distribution of (characteristic polynomial, 1-eigenspace dimension) pairs among Frobenius elements associated to each good prime in $[10000, U]$.
        \begin{enumerate}
            \item For each sampled (characteristic polynomial, 1-eigenspace dimension) pair, keep only the subgroups in \texttt{possibilities} containing an element realizing this invariant.
        \end{enumerate}
        \item Using Algorithm \ref{alg: vector-distances}, compute the set of most likely possible images and a confidence parameter $\eta$ (see Section \ref{sec: matching}, particularly Definition \ref{def: likelihood ratio} for details). 
        \item Increase $U$ by 10000.
    \end{enumerate}
    \textbf{Until} $\eta > \nu$ or $U \geq N$.

    \item Check whether $C$ has a rational 5-torsion point. Store the result as a boolean     \texttt{rational\_point}.
    \item \textbf{If} \texttt{rational\_point} \textbf{then} replace \texttt{possibilities} with the subset of groups in \texttt{possibilities} which have a $(1)$-group eigenspace. \\ \textbf{Else} replace \texttt{possibilities} with the subset of \texttt{possibilities} without a $(1)$-group eigenspace.
    
    \item Using Algorithm \ref{alg: finding_torsion}, try to compute whether $C$ has a simple quadratic 5-torsion point (see Definition \ref{def: simple_quadratic}). Store the result (either \texttt{true}, \texttt{false}, or \texttt{maybe}) as \texttt{simple\_quadratic\_point}.
    \item \textbf{If} \texttt{simple\_quadratic\_point} is \texttt{true} \textbf{then} replace \texttt{possibilities} with the subset of groups in \texttt{possibilities} which have a $(\pm 1)$-group eigenspace \\ 
    \textbf{If} \texttt{simple\_quadratic\_point} is \texttt{false} \textbf{then} replace \texttt{possibilities} with the subset of groups in \texttt{possibilities} without a $(\pm 1)$-group eigenspace.
    \item \textbf{Return} \texttt{possibilities}.
\end{enumerate}
\end{staticalgorithm}

\subsection{Local Information} \label{sec: local}

\begin{defi}[Local distribution (of a group)] Given a finite matrix group $G$ over a field $F$, the \textit{local distribution} of $G$ is a probability mass function $\mathcal{F}_G: F[x] \times \Z \rightarrow \R$ where
\[\mathcal{F}_G(f(x), n) = \frac{1}{\#G} \cdot \#\left\{M \in G : \textrm{CharacteristicPolynomial}(M) = f(x) \textrm{ and 1-eigenspace-dimension}(M) = n \right\}.\]
That is, the local distribution assigns to each (characteristic polynomial, 1-eigenspace dimension) pair the probability of obtaining that pair when sampling elements uniformly from the group $G$.
\end{defi}

\begin{lemma}
    Up to conjugacy, no more than eight subgroups of $\GSp_4(\FF_5)$ with surjective similitude character all share the same local distribution.
    \label{lem: bucket-size}
\end{lemma}

\begin{proof}
    The code used to compute this fact can be found in the \href{https://github.com/maathilde-k/Mod-5-Galois-Images-of-Genus-2-Abelian-Curves/blob/main/Paper%20Lemmas/99pairs.magma}{Paper Lemmas} folder in the Github repository.
\end{proof}

\begin{rmk}
    We will refer to the local distribution $\mathcal{F}_{\im \rho_{J, 5}}$ of the image of Galois associated to $J$ as the \emph{true local distribution} of $J$. This is in contrast to the below defined empirical distribution.
\end{rmk}

\begin{defi}[Empirical local distribution]
    Given a Jacobian $J$ and a prime bound $N$, the \emph{empirical local distribution} $\mathcal{E}(J, N)$ is the distribution of the pair (characteristic polynomial, 1-eigenspace dimension) associated to $\Frob_p$ acting on $J[5]$, when $p$ is drawn from primes between $10000$ and $N$, inclusive.
\end{defi}

\begin{rmk}
    Note that characteristic polynomials and dimensions of $1$-eigenspaces are conjugation invariant, so it is not an issue that Frobenius elements are only defined up to conjugacy.
    \label{rmk: pair well defined}
\end{rmk} 

\begin{rmk}
    By the Chebotarev Density Theorem, the empirical local distribution $\mathcal{E}(J, N)$ converges to the true local distribution $\mathcal{F}_{\im \rho_{J, 5}}$ as $N \rightarrow \infty$. Thus, for large enough $N$, the true distribution is that which is nearest the empirical distribution. This is made precise in Lemma \ref{lem: distribution-filter}.
    \label{rmk: distribution convergence}
\end{rmk}

We make use of this in two steps. Algorithm \ref{alg: frob_sample} computes the empirical local distribution. Then, Algorithm \ref{alg: vector-distances} selects the subgroups with local distributions best resembling the empirical local distribution computed from Algorithm \ref{alg: frob_sample}.

\subsubsection{Sampling Algorithm} \label{sec: sampling}

Algorithm \ref{alg: frob_sample} computes the empirical local distribution by sampling Frobenius elements corresponding to primes in the interval $[a,b]$. Characteristic polynomials are acquired using \texttt{Smalljac} \cite{smalljac} and dimensions of $1$-eigenspaces are obtained by counting $5$-torsion points of the Jacobian over $\Fp$. The algorithm is formalized below.

\begin{staticalgorithm}
    \staticalgorithmcaption{An algorithm to compute the empirical local distribution by sampling Frobenius elements}{A Jacobian $J$ of a genus 2 curve and a range $[a,b]$ over which to sample primes.}{A dictionary \texttt{dist} whose keys are (characteristic polynomial, 1-eigenspace dimension) tuples and whose values are the corresponding empirical frequencies.} \label{alg: frob_sample}

    \begin{enumerate}
        \item Initialize the multiset \texttt{invariant\_counts}.
        \item \textbf{For} all primes $p \in [a,b]$ of good reduction (i.e. not dividing the conductor of $J$), add the tuple ($\Pi_p$, ${\rm dim}_p$) to \texttt{invariant\_counts}, where $\Pi_p$ and ${\rm dim}_p$ are respectively the characteristic polynomial and the dimension of the 1-eigenspace of $\Frob_p$.
        \item Initialize the dictionary \texttt{dist}.
        \item \textbf{For} all pairs $(\Pi, \dim)$ in \texttt{invariants\_count} \textbf{do}
        \begin{enumerate}
            \item Add to \texttt{dist} a key-value pair whose key is $(\Pi, \dim)$ and whose value is the the quotient of the multiplicity of $(\Pi, \dim)$ in \texttt{invariant\_counts} by $(\pi(a) - \pi(b))$, where $\pi$ is the prime counting function.
        \end{enumerate}
        \item \textbf{Return} \texttt{dist}  
    \end{enumerate}
\end{staticalgorithm}

\begin{rmk}
    The output of Algorithm \ref{alg: frob_sample} enables us to eliminate some possible images immediately. For a given subgroup $H$, we check if each realized (characteristic polynomial, 1-eigenspace dimension pair) in the empirical local distribution is realized by an element of $H$. If not, then we may rule out $H$ as a possibility.
\end{rmk}

Running Algorithm \ref{alg: frob_sample} using $[10000, 20000]$ on the example curve, we compute the local distribution presented in the third column of Table \ref{tab: empirical-true-dist}.

\subsubsection{Distribution Matching} \label{sec: matching}

By Lemma \ref{lem: 99-vector space}, there are exactly 99 distinct (characteristic polynomial, 1-eigenspace dimension) pairs. Thus, we embed the space of local distributions into $\R^{99}$ and use the Euclidean metric to measure similarity. Declaring this embedding an isometry defines a metric on the space of local distributions, allowing us to state Remark \ref{rmk: distribution convergence} precisely.

\begin{lemma} \label{lem: distribution-filter}
There exists a constant $N$, depending on $J$, such that for all $n>N$, there is an inequality $|\mathcal{E}(J, n) - \mathcal{F}_{\im \rho_{J, 5}}| \leq |\mathcal{E}(J, n) - \mathcal{F}_G|$ for any candidate subgroup $G \leq \GSp_4(\FF_5)$, with equality only if $\mathcal{F}_{\im \rho_{J, 5}} = \mathcal{F}_G$.
\end{lemma}

\begin{proof}
Let $\mathbf{x}_n$ and $\mathbf{y}$ denote the images of the empirical distribution $\mathcal{E}(J, n)$ and the true distribution $\mathcal{F}_{\im \rho_{J, 5}}$, respectively, under the embedding into $\mathbb{R}^{99}$. Denote their coordinates $x_1, ..., x_{99}$ and $y_1, ..., y_{99}$, respectively. We first show that $\lim_{n \to \infty} \mathbf{x}_n = \mathbf{y}$.

Let $c$ denote the number of conjugacy classes in $\GSp_4(\FF_5)$. Index the axes of $\R^c$ by conjugacy classes $\{\mathcal{C}\}$. Upgrading the notation of Theorem \ref{thm: chebotarev}, let $$P_\mathcal{C}(a, b) = \#\{\text{prime } p : \Frob_p \in \mathcal{C}, p \in [a, b] \}.$$ Let $\mathbf{w}_n, \mathbf{z} \in \R^c$ have coordinates $w_\mathcal{C} = P_\mathcal{C}(10000, n)/(\pi(n)-\pi(10000))$ and $z_\mathcal{C} = \#\mathcal{C}/\#\GSp_4(\FF_5)$, respectively. The Chebotarev density theorem says that $w_\mathcal{C} \rightarrow z_\mathcal{C}$ as $n \rightarrow \infty$ for each $\mathcal{C}$. Thus, $\mathbf{w}_n \rightarrow \mathbf{z}$ as $n \rightarrow \infty$. 

Let $p_\mathcal{C}(\lambda)$ denote the characteristic polynomial of the elements of $\mathcal{C}$, and let $d(\mathcal{C}, \lambda)$ denote the dimension of their 1-eigenspaces. There is a projection map $\R^c \rightarrow \R^{99}$ given by sending $e_\mathcal{C}$ to the basis vector corresponding to the pair $(p_\mathcal{C}, d(M, 1))$. This projection map sends $\mathbf{w}_n \mapsto \mathbf{x}_n$ and $\mathbf{z} \mapsto \mathbf{y}$, so it follows from continuity of the projection that $\mathbf{x}_n \mapsto \mathbf{y}$ as $n \rightarrow \infty$.

Note that since $\GSp_4(\FF_5)$ is a finite group, its subgroup lattice is finite and there are finitely many distinct possibilities $\mathcal{F}_{G}$ which we must consider. This means that there is some minimal positive value
\[b = \min_{G} |\mathcal{F}_{G} - \mathcal{F}_{\im \rho_{J, 5}}|,\]
where $G$ ranges over candidate subgroups. Since $\mathbf{x}_n \to \mathbf{y}$, i.e. $\mathcal{E}(J, n) \to \mathcal{F}_{\im \rho_{J, 5}}$, there exists a constant $M$ such that $n > M$ implies that $|\mathcal{E}(J, n) - \mathcal{F}_{\im \rho_{J, 5}}| < \frac{b}{2}$. The desired statement follows immediately.
\end{proof}

Using this metric, we compute the list of possible images whose local distributions are closest to the empirical distribution. This is summarized in Algorithm \ref{alg: vector-distances}.

\begin{staticalgorithm}
    \staticalgorithmcaption{An algorithm to compute likely images.}{A set \texttt{possibilities} of possible subgroups and a dictionary \texttt{frob\_dist} encoding the empirical local distribution output from Algorithm \ref{alg: frob_sample}.}{The subset \texttt{candidates} of \texttt{possibilities} with minimal Euclidean distance to the empirical distribution $X$ and a parameter $\eta$ indicating the estimated likelihood that \texttt{candidates} indeed contains $\im \rho_{J, 5}$.} \label{alg: vector-distances}
    
    \begin{enumerate}
        \item \textbf{For} each \texttt{subgroup} in \texttt{possibilities} \textbf{do} 
        \begin{enumerate}
            \item Compute the local distribution $\mathcal{F}_G$ associated to \texttt{subgroup}.
            \item Treating the empirical distribution \texttt{frob\_dist} and the subgroup's distribution $\mathcal{F}_G$ as vectors in $\mathbb{R}^{99}$, compute the Euclidean distance between the two vectors.
        \end{enumerate} 
        \item Set \texttt{best\_subgroups} to be the set of subgroups whose corresponding distributions achieve the minimum observed distance to the empirical distribution, and \texttt{second\_subgroups} to be the set of subgroups whose distributions in achieve the second lowest computed value.

        \item Compute the log-likelihood difference (see Definition \ref{def: likelihood ratio}) $\eta = \mathcal{L}(J, n)$.
        
        \item \textbf{Return} \texttt{candidates} and $\eta$.
    \end{enumerate}
\end{staticalgorithm}

\begin{rmk}
    By construction, Algorithm \ref{alg: vector-distances} will return a list of possible subgroups where all elements in the list have the same local distribution. Thus, by Remark \ref{rmk: order}, all subgroups in the list will have equal order, and by Lemma \ref{lem: bucket-size}, the list will have length at most eight.
\end{rmk}

\begin{lemma}
    Fix a Jacobian $J$ and let \texttt{candidates} be the output of running Algorithms \ref{alg: frob_sample} and \ref{alg: vector-distances} on $J$ and $[10000, n]$. Then there exists $N \in \mathbb{N}$, depending on $J$, such that for all $n > N$, the image $\im \rho_{J, 5}$ is an element of \texttt{candidates}.
    \label{lem: local correctness}
\end{lemma}

\begin{proof}
    Take $N$ to be the value from the proof of Lemma \ref{lem: distribution-filter}. Since Algorithm \ref{alg: vector-distances} picks out the subgroups which minimize the Euclidean distance to the empirical local distribution, $\im \rho_{J, 5}$ will be an element of \texttt{candidates}.
\end{proof}

In our example, applying Algorithm \ref{alg: vector-distances} returns subgroups 5.624.2 and 5.624.4. These both have the same local distribution. Putting this next to the empirical distribution obtained earlier, we get Table \ref{tab: empirical-true-dist}.

\begin{table}[h]
\centering

\begin{tabular}{|  c  |c  |c |c|}
\hline
Characteristic Polynomial & $1$-eigenspace dimension & Empirical Proportion &5.624.2 and 5.624.4 Proportion\\
\hline
$x^4 + 4$& $1$& $12.88\%$ &$12.50\%$\\
\hline
$x^4 + 3x^2 + 1$& $1$& $9.58\%$ &$9.17\%$\\
\hline
$x^4 + 3x^2 + 1$& $2$& $2.32 \%$ &$2.29\%$\\
\hline
$x^4 + x^3 + x^2 + x + 1$ & $1$ & $4.16\%$ &$4.00\%$\\
\hline
$x^4 + x^3 +x^2 + x + 1$ & $2$ & $0.97\%$ &$1.00\%$\\
\hline
$x^4 + x^3 + 3x^2 + 2x + 4$ & $0$ & $6.58 \%$ &$6.25\%$\\
\hline
$x^4 + x^3 + 3x^2 + 4x+1$ & $1$ & $6.49\%$ &$5.21\%$\\
\hline
$x^4 + x^3 + 4x^2 + 3x+4$ & $0$ & $4.36\%$ &$4.17\%$\\
\hline
$x^4 + 2x^3 + x^2 + 4x + 4$ & $0$ & $3.58\%$ &$4.17\%$\\
\hline
$x^4 + 2x^3 + 3x^2 + x+4$ & $0$ & $4.07\%$ &$4.17\%$\\
\hline
$x^4 + 2x^3 + 3x^2 + 3x+1$ & $1$ & $2.52\%$ &$4.17\%$\\
\hline
$x^4 + 2x^3 + 4x^2 + 2x+1$ & $1$ & $3.58\%$ &$3.33\%$\\
\hline
$x^4 + 2x^3 + 4x^2 + 2x+1$ & $2$ & $0.97\%$ &$0.83\%$\\
\hline
$x^4 +3x^3 + 2x^2 + 3x+1$ & $1$ & $4.94\%$ &$5.00\%$\\
\hline
$x^4 + 3x^3 + 2x^2 + 3x+1$ & $2$ & $0.68\%$ &$1.25\%$\\
\hline
$x^4  +3x^3 + 2x^2 + 4x + 4$ & $0$ & $5.91\%$ &$6.25\%$\\
\hline
$x^4 + 3x^3 +3x^2 + 2x + 1$ & $1$ & $4.65\%$ &$4.17\%$\\
\hline
$x^4 +3x^3 + 4x^2 + x + 4$ & $0$ & $4.45\%$ &$4.17\%$\\
\hline
$x^4 + 4x^3 + 4x + 1$ & $1$ & $2.81\%$ &$3.33\%$\\
\hline
$x^4 + 4x^3 + 4x + 1$ & $2$ & $0.97\%$ &$0.83\%$\\
\hline
$x^4 + 4x^3 + x^2 + 2x + 4$ & $0$ & $3.58\%$ &$4.17\%$\\
\hline
$x^4 + 4x^3 + 2x^2 + 3x + 4$ & $0$ & $4.94\%$ &$4.17\%$\\
\hline
$x^4 + 4x^3 + 3x^2 + x + 1$ & $1$ & $5.03\%$ &$5.21\%$\\
\hline
$x^4 + x^3 + x^2 + x + 1$ & $3$ & $0\%$ &$0.21\%$\\
\hline
$x^4 + x^3 + x^2 + x + 1$ & $4$ & $0\%$ &$<0.01\%$\\
\hline
\end{tabular}
\caption{Comparing the empirical local distribution and the local distribution of the closest matches.}
\label{tab: empirical-true-dist}
\end{table}

Thus at this point we assume that the set of possible images consists of $\{5.624.2, 5.624.4\}$. While the methods used for obtaining this local information is probabilistic and only guaranteed to be correct asymptotically, the data are empirically promising. If we assume that Frobenius elements are sampled independently at random from a uniform distribution on elements of the true subgroup, then the probability of obtaining the observed distribution from either $5.624.2$ or $5.624.4$ is $\exp(-3056.318)$. In contrast, the subgroups next most likely to yield the observed distribution are are $5.312.1$ and $5.312.2$ (both have the same local distribution), and their corresponding probability is $\exp(-3461.531)$. The large discrepancy in likelihoods -- a few hundred orders of magnitude -- gives us confidence that the information derived in the procedure is accurate. 

It is important to note that some distinct local distributions are very similar. For instance, the Euclidean distance between the local distributions of subgroups 5.74880.13 and 5.374400.24 is only $\frac{2}{15625}$. Our algorithm requires a comparatively large number of sampled primes to distinguish such subgroups from one another. A desire to account for this motivates the condition in Step 3 of Algorithm \ref{alg: imager} where we sample additional Frobenius elements until we reach a desired log-likelihood difference. 

\begin{defi}(log-likelihood ratio)
    Denote by $\{\mathcal{D}_i\}$ the set of local distributions corresponding to possible images of Galois. Let $\pi(n)$ denote the number of primes up to $n$. For $J$ a Jacobian of a genus 2 curve and $N>10000$, let $\mathcal{P}(\mathcal{D}_i, J, N)$ denote the probability, when sampling $\pi(N) - \pi(10000)$ independent, identically distributed observations from $\mathcal{D}_i$, of obtaining a sample perfectly representative of the empirical local distribution $\mathcal{E}(J, N)$. Let $m$ denote the value of the index $i$ for which $\mathcal{P}(\mathcal{D}_i, J, N)$ is maximized. The \emph{log-likelihood ratio} $\mathcal{L}(J, N)$ for $J$ and $N$ may then be computed as 
    \[
        \mathcal{L}(J, N) = \log (\mathcal{P}(\mathcal{D}_m, J, N)) - \log\left(\max_{i \neq m}\mathcal{P}(\mathcal{D}_i, J, N) \right).
    \]
    \label{def: likelihood ratio}
\end{defi}
\noindent Using this notation, in our example Jacobian $J$ we have
\[\mathcal{L}(J, 20000) = (-3056.318)-(-3461.531) = 405.213.\]

\subsection{Global Information} \label{sec: global}

\subsubsection{Rational 5-torsion} \label{sec: rational}

The following lemma provides a strategy for filtering the list of possibilities for $\im \rho_{J,5}$.

\begin{lemma} \label{lem: rational-filter}
    An abelian surface $J$ has a nontrivial rational 5-torsion point if and only if $\im \rho_{J, 5}$ pointwise fixes a one-dimensional subspace $W \subset (\FF_5)^4$, i.e. if $\im \rho_{J, 5}$ has a (1)-group eigenspace (recall Definition \ref{def: group_eigenspace}).
\end{lemma}

If $J$ has a nontrivial rational 5-torsion point then we filter for the set of subgroups which have a $(1)$-group eigenspace. Otherwise we filter for the set of subgroups which lack a $(1)$-group eigenspace.  
A first algorithm to compute the torsion part of the Mordell-Weil group is provided by Stoll \cite[pp.~198--201]{height-bounds}, and there is an implementation in Magma \cite{Magma}.

In our running example, $J(\Q) \cong \Z$. In particular, $J(\Q)[5]$ is trivial, so $\im \rho_{J, 5}$ cannot fix a 1-dimensional subspace. Since neither 5.624.2 nor 5.624.4 fix a 1-dimensional subspace, this step does not narrow down the set of possibilities in our example. After checking for rational torsion, the set of possibilities remains $\{5.624.2, 5.624.4\}$.

\subsubsection{Quadratic 5-torsion} \label{sec: quadratic}

\begin{defi}[Simple quadratic 5-torsion]
    Let $P \in J[5]$ be a $5$-torsion point of the form $[P_1]+[P_2] - [\infty_1]-[\infty_2]$ for affine $P_1, P_2 \in C$. Let $P = (x^2 + \alpha x + \beta, \gamma x + \delta)$ denote the Mumford coordinates of $P$. We say that $P$ is a \emph{simple quadratic 5-torsion point} if $\alpha$ and $\beta$ are rational, and $(\gamma, \delta) = (a\sqrt{d}, b\sqrt{d})$ for $a, b \in \Q$ and $d \in \Z$ square-free. 
    \label{def: simple_quadratic}
\end{defi}

\begin{lemma}[Galois Orbit of Simple Quadratic 5-torsion]
    As above, let $P \in J[5]$ have maximum degree Mumford coordinates. Then, the orbit of $P$ under the action of Galois is $\{\pm P\}$ if and only if $P$ is a simple quadratic 5-torsion point.
    \label{lemma:simple-quad-orbits}
\end{lemma}

\begin{proof}
    If $P$ is a simple quadratic 5-torsion point, then its Mumford coordinates are of the form $(x^2+\alpha x+\beta, (a\sqrt{d})x+ b\sqrt{d})$ for $\alpha, \beta, a, b \in \Q$. Its orbit under Galois is then easily seen to be $\{\pm P\}$. Conversely, if $P = (x^2 + \alpha x + \beta, \gamma x + \delta)$ has orbit $\{\pm P\}$ under Galois, then the orbits of its coordinates are $\{\alpha\}$, $\{\beta\}$, $\{\pm\gamma\}$, and $\{\pm\delta\}$. Thus, $\alpha$ and $\beta$ are rational. The coefficients $\gamma$ and $\delta$ are fixed by the same automorphisms of $\overline{\Q}$, so they lie in the same minimal Galois extension $K/\Q$. Since their orbits are size 2, $\gamma$ and $\delta$ have degree 2 minimal polynomials, so the extension $K/\Q$ is quadratic. Then, $K = \Q(\sqrt{d})$ for some square-free $d \in \Z$. Let $\sigma$ be the non-trivial automorphism of $K$. The kernel of $(\sigma + {\rm {Id}})$ is exactly $\{a\sqrt{d} : a \in \Q\}$, so $\gamma$ and $\delta$ are of that form.
\end{proof}

\begin{lemma}[Simplification is lossless]
    Let $J$ be a Jacobian of a genus 2 hyperelliptic curve defined over $\Q$, with no rational 5-torsion but $J(\Q(\sqrt{d}))[5]$ non-trivial for some $d$. Then there is a simple quadratic 5-torsion point $P \in J(\Q(\sqrt{d}))[5]$. In other words, to determine whether there is 5-torsion in a quadratic field, it suffices to look for simple quadratic torsion.
    \label{lemma:simple-quadratic-suffices}
\end{lemma}

\begin{proof}
   With the context of the lemma, let $J(\Q(\sqrt{d}))[5] \cong (\Z/5\Z)^n$ for some $n\geq 1$. The automorphism $\sigma: \sqrt{d}\rightarrow -\sqrt{d}$ acts linearly as a non-trivial involution $M_\sigma$ on $(\Z/5\Z)^n$. Since $M_\sigma^2-I=0$ but $M_\sigma - 1 \neq 0$, we find that $M_\sigma + 1$ divides the minimal polynomial of $M_\sigma$, and thus that there is some $-1$ eigenvector $P$. The action of $\Gal(\overline{\Q}/\Q)$ on $J(\Q(\sqrt{d}))$ factors through $\Gal(\Q(\sqrt{d})/\Q)$, so $P$ is a 5-torsion point with orbit $\{\pm P\}$. If $P$ has full-degree Mumford coordinates, then $P$ is a simple quadratic torsion point by Lemma \ref{lemma:simple-quad-orbits}. If not, then $2P$ does, and so is a simple quadratic torsion point.
\end{proof}

\begin{cor} The image $\im \rho_{J, 5}$ has a $(\pm 1)$-group eigenspace if and only if $J$ has simple quadratic torsion. 
\label{cor: correctness-of-simple-quad}
\end{cor}

\begin{proof}
Let $\textrm{Span}(P)$ be a $(\pm 1)$-group eigenspace of $J(\overline{\Q})[5]$. Then, under every Galois automorphism, $P \mapsto P$ or $P \mapsto -P$. By Lemma \ref{lemma:simple-quad-orbits} this tells us that $P$ is a simple quadratic $5$-torsion point. Conversely, suppose $P \in J(\overline{\Q})[5]$ is a simple quadratic torsion point. Again by Lemma \ref{lemma:simple-quad-orbits}, this implies that the orbit of $P$ under the action of Galois is $\{\pm P\}$. This means that every element of the Galois group has $P$ as either a $1$-eigenvector or a $(-1)$-eigenvector, so the image of Galois has a $(\pm 1)$-group eigenspace, namely $\textrm{Span}(P)$.
\end{proof}

Corollary \ref{cor: correctness-of-simple-quad} gives a necessary and sufficient condition for restricting the possible images under consideration. We thus perform a procedure analogous to the one outlined in Section \ref{sec: rational}. We first check if $J$ has simple quadratic torsion points. If it does, we filter for subgroups with a $(\pm 1)$-group eigenspace. If $J$ lacks simple quadratic torsion then we filter for subgroups without a $(\pm 1)$-group eigenspace.

We determine if $\Jac(C)$ has simple quadratic torsion in two phases. In the first phase we use information about the Jacobian mod $p$ to produce a finite list of quadratic extensions of $\Q$ which could contain simple quadratic $5$-torsion points. In the second phase we search for concrete instances of simple quadratic $5$-torsion.

\begin{rmk}
    By the Ner\'on-Ogg-Shafarevich criterion, any torsion growth field has ramification only at primes of bad reduction. This immediately gives the following lemma.
\end{rmk}

\begin{lemma}
    \label{lem: limited_dees}
    Let $C$ be a genus 2 curve with Jacobian $J = \Jac(C)$ and conductor $N$. Let $L = \Q(\sqrt{d})$. Let $P \in J(L)$ be a simple quadratic torsion point. Then, we have $\textrm{rad}(\textrm{Disc}(L)) \,|\, N$.
\end{lemma}

For any $N$, there are only finitely many quadratic fields with $N$-smooth discriminant. Thus, Lemma \ref{lem: limited_dees} allows us to compute a finite list of fields which could possibly contain simple quadratic 5-torsion. This list can be narrowed further using the following theorem.

\begin{thm} \label{thm: quadratic-elimination}
    Let $P \in J$ be a simple quadratic 5-torsion point defined over $K = \Q(\sqrt{d})$. Let $p \neq 5$ be a good rational prime which splits in $K$. Then, $J/\Fp$ admits a rational 5-torsion point.
\end{thm}
\begin{proof}
    Let $K$ be a quadratic field for which $J(K)[5] \neq 0$. Let $p \neq 5$ be a prime of good reduction with $p\O_K = \p_1\p_2$ for $\p_1, \p_2 \in \Spec^*(\O_K)$. There is a reduction-mod-$\p$ map $J(K) \rightarrow J(\FF_\p) \cong J(\Fp)$. Since the multiplication-by-5 map on $J/\FF_\p$ is \'etale, Hensel's lemma guarantees that reduction modulo $\p$ is injective on 5-torsion.
\end{proof}

The contrapositive of Theorem \ref{thm: quadratic-elimination} allows us to sample primes and use them to preclude certain quadratic fields from containing simple quadratic 5-torsion. The first phase of the simple quadratic torsion determination is summarized in Algorithm \ref{alg: quad_phase_one}. 

\begin{staticalgorithm}
\staticalgorithmcaption{An algorithm to compute where quadratic 5-torsion is plausible.}{A hyperelliptic curve $C$ with conductor $N$.}{A list of integers $d$, containing all $d$ such that $C$ may have 5-torsion over $\Q(\sqrt{d})$.} \label{alg: quad_phase_one}

\begin{enumerate}
    \item Initialize \texttt{possibilities} to the set of integers $d$ such that $\textrm{Disc}(\Q(\sqrt{d})) \,|\, N^k$ for some $k > 0$.
    \item \textbf{For} good prime $p \in [10,000,\; 20,000]$ or \textbf{until} \texttt{possibilities} is empty \textbf{do}
    \begin{enumerate}
        \item Compute $\Jac(C)(\Fp)$, the group of rational points of the Jacobian over $\Fp$.
        \item \textbf{If} 5 does not divide $\#\Jac(C)(\Fp)$ \textbf{then}
            \begin{enumerate}
                \item Remove from \texttt{possibilities} each $d$ for which $p$ splits in $\Q(\sqrt{d})$. 
            \end{enumerate}
    \end{enumerate}
    \item \textbf{Return} \texttt{possibilities}.
\end{enumerate}
\end{staticalgorithm}
\

If Algorithm \ref{alg: quad_phase_one} produces an empty list, we conclude there is no simple quadratic 5-torsion. If Algorithm \ref{alg: quad_phase_one} produces a non-empty list, then the second phase attempts to find simple quadratic 5-torsion in the remaining fields. In our example curve, running Algorithm \ref{alg: quad_phase_one} returns that there is one quadratic extension where there might be additional $5$-torsion, $\Q(\sqrt{5})$. Thus we will check for additional $5$-torsion over $\Q(\sqrt{5})$.

Jacobians of genus 2 curves can be thought of in two very distinct ways. One way is as a quotient $\C^2/\Lambda$ of $\C^2$ by a lattice, with points on the Jacobian represented as points in a fundemental domain. The other way is that which we have used so far in this paper: a blowdown of $\textrm{Sym}^2(C)$, with points represented via Mumford coordinates. In Magma, these representations are separate objects, referred to as analytic and algebraic Jacobians, respectively. These separate objects are related by the Abel-Jacobi map, which maps from the algebraic Jacobian to the analytic.

To find concrete instances of simple quadratic torsion in phase 2, we start by computing the $5$-torsion points on the analytic Jacobian, realized via high-precision floating point values. These $5$-torsion points are then pulled back to the algebraic Jacobian, at which point we seek to determine if the Mumford coordinates match the form needed for simple quadratic torsion.

Let $P \in \Jac(C)(\C)$ be a 5-torsion point defined over $\C$ with Mumford coordinates $P : (x^2 + \alpha x + \beta, \gamma x + \delta)$. If $P$ is simple quadratic defined over $K = \Q(\sqrt{d})$, then $\alpha$, $\beta$, $\gamma/\sqrt{d}$, and $\delta/\sqrt{d}$ are all rational. We determine if the pair of floating point numbers representing a complex number likely represents a rational number using its continued fraction expansion, which is available via the Magma function \texttt{ContinuedFraction}. 
If the complex coordinates of a point $P$ can be identified with algebraic values of bounded height, a point $\Tilde{P}$ with those coordinates is constructed in $\Jac(C)(K)$ and it is checked whether $5\Tilde{P}$ is the identity.\footnote{This check is performed to guard against cases where the coordinates of a 5-torsion point are coincidentally extremely close to a rational number.} The entire procedure is summarized in Algorithm \ref{alg: finding_torsion}. \\

\begin{staticalgorithm}
\staticalgorithmcaption{An algorithm to compute whether the Jacobian of a hyperelliptic curve has simple quadratic 5-torsion.}{A hyperelliptic curve $C$}{A ``boolean" - \texttt{true} if the curve $C$ has a simple quadratic 5-torsion point, \texttt{false} if the algorithm can prove a lack of simple quadratic 5-torsion, and \texttt{maybe} otherwise.} \label{alg: finding_torsion}

\begin{enumerate}
    \item Initialize \texttt{possibilities} to be the output of Algorithm \ref{alg: quad_phase_one}, with input $C$. This is the list of quadratic extensions over which $J=\Jac(C)$ possibly has 5-torsion.

    \item \textbf{If} \texttt{possibilities} is an empty list \textbf{then} \textbf{return} \texttt{false}, since there are no candidate quadratic extensions.

    \item Compute a bound $h$ on the na\"ive height of torsion in $\Q(\sqrt{d})$ (see Section \ref{sec: height_bounds}). 

    \item Construct the analytic Jacobian, which is represented as $\C^2 / \Lambda$ for a lattice $\Lambda$. Compute a $\Z$-basis $v_1,v_2,v_3,v_4$ for $\Lambda$. The $5$-torsion points of the analytic Jacobian are linear combinations of $\frac{1}{5}v_1$, $\frac{1}{5}v_2$, $\frac{1}{5}v_3$, $\frac{1}{5}v_4$. Store one representative for each 1-D subspace of $J(\C)[5]$ in the list \texttt{all\_torsion}. 

    \item Map each element of \texttt{all\_torsion} to the algebraic Jacobian of $C$. Call this new list \\ \texttt{algebraic\_torsion\_points}. 
    
    \item Initialize the empty list \texttt{good\_torsion}.
    \item \textbf{For} each \texttt{point} in \texttt{algebraic\_torsion\_points} \textbf{do}
    \begin{enumerate}
        \item Considering \texttt{point} as $P = (u, v) = (x^2+\alpha x + \beta, \gamma x + \delta)$, compute whether $\alpha$ and $\beta$ have rational approximations with height bounded by $h$ and error bounded by $\varepsilon =\frac{1}{2e^{2h}}$. Our choice of $\varepsilon$ ensures that the approximations to $\alpha$ and $\beta$ are unique, provided they exist.
        \item \textbf{If} such approximations $\alpha \approx a_n/a_d$ and $\beta \approx b_n/b_d$ are found \textbf{then}
        \begin{enumerate}
            \item Compute a list \texttt{upoints} of the four points with mumford coordinates $(x^2+\frac{a_n}{a_d} x + \frac{b_n}{b_d}, \gamma_i x + \delta_i)$.
            \item \textbf{For} $P_i$ in \texttt{upoints} \textbf{do}
            
            \quad \quad \textbf{If} $\gamma_i$ and $\delta_i$ are of the form $m\sqrt{d}$ and $n \sqrt{d}$ respectively, for $m, n \in \Q$ \textbf{then} 

            \quad\quad\quad\quad append $P_i$ to \texttt{good\_torsion}.
        \end{enumerate}
        
    \end{enumerate}
    \item \textbf{For} each \texttt{point} in \texttt{good\_torsion} \textbf{do} 
    
    \quad \quad \textbf{If} \texttt{5*point=Id} \textbf{then} \textbf{return} \texttt{true}. 
        
    \item \textbf{return} \texttt{maybe}
\end{enumerate}
\end{staticalgorithm}

\begin{rmk} \label{rmk: false-negative}

     A word of caution: step 5 of Algorithm \ref{alg: finding_torsion} relies on accurate inversion of the Abel-Jacobi map. However, the Magma function \texttt{FromAnalyticJacobian} is known to not always exhibit this kind of numerical stability; see Example 3.4.9 in \cite{costa2017rigorous}. In other words, phase 2 produces no false positives but may produce false negatives. This is why the algorithm outputs \texttt{maybe} at the end, rather than \texttt{false}. Since we cannot bound the error under \texttt{FromAnalyticJacobian}, we heuristically use $\varepsilon = \frac{1}{2e^{2h}}$.

     In order to produce a more robust algorithm, one would need a numerically stable inversion of Abel-Jacobi where one could bound the numerical error in inverting a given point. If we had such an algorithm, then Algorithm \ref{alg: finding_torsion} could be edited to provably compute whether a curve admits simple quadratic 5-torsion.
\end{rmk}

\begin{rmk}
    Such an algorithm would also allow one to easily adapt Algorithm \ref{alg: finding_torsion} to rigorously compute torsion subgroups over imaginary quadratic fields.
\end{rmk}

Algorithm \ref{alg: finding_torsion}, using a hypothetical \texttt{StableAbelJacobiInversion} instead of  \texttt{FromAnalyticJacobian}, is guaranteed to be correct, as it exhaustively checks all potential $5$-torsion points. The properties needed for a guarantee of correctness are a bound on the error of the inversion and on the height of the rational approximations. Promisingly, there are indeed algorithms that stably invert Abel Jacobi with computably bounded error \cite[pp.~5--11]{costa2017rigorous}.

\begin{lemma}\label{lem: quad-tors}
    If Algorithm \ref{alg: finding_torsion} returns \texttt{true} then $\Jac(C)$ contains a simple quadratic torsion point.
\end{lemma}

\begin{proof} 
    This follows from step 8 of Algorithm \ref{alg: finding_torsion}.
\end{proof}

In our running example, we find that there is additional $5$-torsion over $\Q(\sqrt{5})$. Specifically, Algorithm \ref{alg: finding_torsion} yields that the point $(x^2-2x+1,\sqrt{5}x-2\sqrt{5})$ is a $5$-torsion point over $\Q(\sqrt{5})$.
This tells us that we can narrow down the list of possible images to subgroups compatible with simple quadratic torsion. Checking each of the current candidates, $\{5. 624.2, 5.624.4\}$, we find that only $5.624.2$ is compatible with simple quadratic torsion, so we conclude that the image is $5.624.2$.

\subsection{Proof of Correctness}

Recall the Main Theorem.
\begin{introthmm}[Main Theorem]
    Let $C/\Q$ be a genus 2 hyperelliptic curve with Jacobian $J$. There is an effective constant $N$, depending only on $C$, such that when sampling all primes in the range $[10000, N]$, Algorithm \ref{alg: imager} produces a list of at most eight equal-order subgroups that contains the mod-$5$ image of Galois.
\end{introthmm}

\begin{proof}
    Algorithm \ref{alg: imager} proceeds by applying successive filters to a list of possible images. It suffices to verify that each filter retains the image $\im \rho_{J, 5}$, and that one filter returns only up to eight subgroups, all of the same order. 

    By Lemma \ref{lem: local correctness}, for adequately large $N$ the output of Algorithm \ref{alg: vector-distances} contains the image of Galois. Furthermore, by Remark \ref{rmk: order}, all subgroups in the output are equal-order, and by Lemma \ref{lem: bucket-size}, there are at most eight. The correctness of the rational torsion filter is proven in Lemma \ref{lem: rational-filter}. 
    Finally, in step 6 we filter based on the output of Algorithm \ref{alg: finding_torsion}. Consider three cases.
    \begin{enumerate}
        \item If Algorithm \ref{alg: finding_torsion} returns \texttt{maybe} then we simply keep \texttt{possibilities} as is. Since \texttt{possibilities} contained the true image before this step, it continues to contain the true image.
        \item If Algorithm \ref{alg: finding_torsion} returns \texttt{true}, then, by Lemma \ref{lem: quad-tors}, $J$ admits simple quadratic torsion. By Corollary \ref{cor: correctness-of-simple-quad}, this corresponds to when the the image of the Galois representation has a $(\pm 1)$-group eigenspace, which is precisely the subset of \texttt{possibilities} which we return. Thus, in this case, the new instance of \texttt{possibilities} coming out of step 6 contains the true image.
        \item Finally, if Algorithm \ref{alg: finding_torsion} returns \texttt{false}, then, by Theorem \ref{thm: quadratic-elimination}, $J$ lacks simple quadratic 5-torsion. Like in the previous case, we may conclude from Corollary \ref{cor: correctness-of-simple-quad} that Algorithm \ref{alg: finding_torsion} does not filter out the image of Galois.
    \end{enumerate}
    
    Since every step returns a list which contains the true image, the whole algorithm must return a list which contains the true image.
\end{proof}

\section{Results} \label{sec: results}

We ran Algorithm \ref{alg: imager} on all genus 2 curves $C$ in LMFDB \cite{lmfdb} for which $\rho_{\Jac(C), 5}$ is not surjective, with prime bound $N=100,000$ and likelihood difference $\nu = 30$ (see Definition \ref{def: likelihood ratio}).

This means we sampled $\Frob_p$ for all good primes $p \in  [10000, 20000]$, a set of size up to 1033, and then, if needed to obtain our desired probability of error, sampled additional primes up to 100,000. In total, we ran Algorithm \ref{alg: imager} on 3990 curves. On a server equipped with an Intel Core i9-12900K processor (16 cores, 24 threads) with a maximum clock speed of 5.2 GHz, the computation took $110$ minutes for the typical curves and $345$ minutes for the atypical curves. As previously mentioned, choosing a value of $N$ large enough to attain the bounds required for an effective version of the Chebotarev Density Theorem would be prohibitively computationally expensive, hence our default choice of only sampling primes $p \in [10000,20000]$. This choice is consistent with prior work, such as \cite{Vogt-surjectivity}. The code used for this computation can be found in the \texttt{LMFDB\_imaging.magma} file in the GitHub repo. We now justify our choice of $N$ by estimating the likelihood of error. 

\begin{lemma}
\label{lem: probability-error-bound}
    Let $J$ be the Jacobian of a genus 2 curve listed in the LMFDB. Model the group elements $\{\Frob_p\}$ as a set of independent random variables drawing uniformly from the image of Galois. Under this model, and assuming a uniform prior probability distribution on possible subgroups, the probability that our algorithm's output on $J$ fails to return a set containing $\im \rho_{J, 5}$ is bounded above by $1.051\cdot 10^{-10}$, or 1 in 9.5 billion.
\end{lemma}

\begin{proof}
    Algorithm \ref{alg: imager} stops sampling Frobenius elements for a given curve when either (i) it achieves a log-likelihood difference of at least 30 or (ii) it samples all primes between 10,000 and 100,000. Crucially for proof of this lemma, when we ran this algorithm, every single curve in fact terminated its sampling because it satisfied the log-likelihood condition. 
    
    Let $J$ be the Jacobian a genus 2 curve in the LMFDB with non-surjective mod-5 Galois representation. Let $N$ be the upper bound on the primes sampled when running Algorithm \ref{alg: imager} on $J$. Let $S = \{\mathcal{D}_i\}$ denote the set of local distributions hailing from possible images, with $\mathcal{D}_1$ denoting the local distribution corresponding to the output of Algorithm \ref{alg: imager} on $J$. Let $G(\mathcal{D}_i)$ denote the set of possible images with local distribution $\mathcal{D}_i$. Using the notation of the statement of Definition \ref{def: likelihood ratio}, we have by the previous paragraph that $\mathcal{P}(\mathcal{D}_1, J, N) \geq e^{30}\mathcal{P}(\mathcal{D}_i, J, N)$ for all $i\neq 1$.

    Under the uniform random model of $\Frob_p$, the quantity $\mathcal{P}(\mathcal{D}_i, J, N)$ is the probability that a Jacobian $J$ would yield a given local distribution, conditional on $\im \rho_{J, 5} \in G(\mathcal{D}_i)$. To bound the probability that $\im \rho_{J, 5} \not\in G(\mathcal{D}_1)$, conditional on the sampled local distribution $\mathcal{E}(J, N)$, apply Bayes' rule.
    
    {\allowdisplaybreaks
    \begin{align*}
        &\PP(\im \rho_{J, 5} \not\in G(\mathcal{D}_1) \>|\> \text{empirical data}) = \frac{\PP(\text{empirical data}\>|\> \im \rho_{J, 5} \not\in G(\mathcal{D}_1))\cdot \PP(\im \rho_{J, 5}\not\in G(\mathcal{D}_1))}{\PP(\text{empirical data})} \\
        & = \frac{\left(\sum_{i \neq 1} \PP(\im \rho_{J, 5} \in \mathcal{D}_i)\cdot\PP(\text{empirical data}\>|\> \im \rho_{J, 5} \in G(\mathcal{D}_i))\right)\cdot (1-\PP(\im \rho_{J, 5} \in G(\mathcal{D}_1))}{\sum_i \PP(\im \rho_{J, 5} \in \mathcal{D}_i)\cdot\PP(\text{empirical data}\>|\> \im \rho_{J, 5} \in G(\mathcal{D}_i))} \\
        & = \frac{\left(\sum_{i \neq 1}  \frac{\#G(\mathcal{D}_i)}{1125}\mathcal{P}(\mathcal{D}_i, J, N)\right)(1-\frac{\#G(\mathcal{D}_1)}{1125})}{\frac{\#G(\mathcal{D}_1)}{1125}\mathcal{P}(\mathcal{D}_1, J, N) + \sum_{i \neq 1} \frac{\#G(\mathcal{D}_i)}{1125}\mathcal{P}(\mathcal{D}_i, J, N)} \\
        & \leq \frac{\left(\sum_{i \neq 1}  \frac{\#G(\mathcal{D}_i)}{1125}e^{-30}\mathcal{P}(\mathcal{D}_1, J, N)\right)(1-\frac{\#G(\mathcal{D}_1)}{1125})}{\frac{\#G(\mathcal{D}_1)}{1125}\mathcal{P}(\mathcal{D}_1, J, N) + \sum_{i \neq 1} \frac{\#G(\mathcal{D}_i)}{1125}e^{-30}\mathcal{P}(\mathcal{D}_1, J, N)} \\
        & = \frac{e^{-30}\left(1 - \frac{\#G(\mathcal{D}_1)}{1125}\right)^2}{\frac{\#G(\mathcal{D}_1)}{1125}+ e^{-30}\left(1 - \frac{\#G(\mathcal{D}_1)}{1125}\right)} \leq \frac{e^{-30}\left(1 - \frac{1}{1125}\right)^2}{\frac{1}{1125}+ e^{-30}\left(1 - \frac{1}{1125}\right)} \approx 1.05 \cdot 10^{-10}.
    \end{align*}
    }
\end{proof}

In Algorithm \ref{alg: finding_torsion}, one must choose a floating point precision with which to carry out analytic Jacobian calculations. We make the heuristic choice to use $D= \min(200, \lceil 5h(C) \rceil)$ decimal digits of precision, where $h$ is a bound (depending on the curve $C$) on the na\"ive height of quadratic torsion (see Algorithm \ref{alg: finding_torsion} and Section \ref{sec: height_bounds}).

For purposes of analysis, we split the data by (geometric) endomorphism algebra. Extra endomorphisms restrict the image of Galois, so sorting curves accordingly is a natural choice. Before detailing the results of our computations, we explain the significance of various common image labels.

\subsection{Subgroups and labels} \label{sec:subgroups_and_labels}

Sutherland's \texttt{GSPLattice} program computes the lattice of subgroups of $\GSp_4(\FF_5)$ with surjective similitude character, up to conjugacy. The program labels each subgroup in accordance with the LMFDB labeling system. Figure \ref{fig:subgroup-lattice} gives a sublattice featuring the few most common images associated to the Jacobians having each class of endormorphism ring. We now describe each subgroup appearing in Figure \ref{fig:subgroup-lattice}.

\begin{figure}[!ht]
    \centering
    \begin{tikzpicture}
        \node at (3, 3) (1d1) {5.1.1};
        
        \node at (3, 2) (156d1) {5.156.1};
        \node at (2, 1) (312d1) {5.312.1};
        \node at (4, 1) (312d2) {5.312.2};
        \node at (0, 0) (624d1) {5.624.1};
        \node at (2, 0) (624d2) {5.624.2};
        \node at (4, 0) (624d3) {5.624.3};
        \node at (6, 0) (624d4) {5.624.4};

        \node at (-3, 2) (325d1) {5.325.1};
        \node at (-3, 1) (650d1) {5.650.1};
        \node at (-6, 0) (3250d1) {5.3250.1};
        \node at (-6, -1) (9750d2) {5.9750.2};
        \node at (-6, -2) (19500d7) {5.19500.7};
        \node at (-4, 0) (6500d2) {5.6500.2};

        \node at (9, 2) (300d1) {5.300.1};
        \node at (9, 1) (600d2) {5.600.2};

        \node at (8, 0) (624d8) {5.624.8};

        \node at (-2, 0) (3900d1) {5.3900.1};
        \node at (-2, -1) (15600d3) {5.15600.3};
        \node at (0, -1) (15600d5) {5.15600.5};

        \node at (0, 2) (9750d1) {5.9750.1};
        \node at (-5, 2) (6500d1) {5.6500.1};
        \node at (-5, 1) (13000d5) {5.13000.5};

        \draw (1d1)--(156d1)--(312d1)--(624d1)--(15600d3);
        \draw (312d1)--(624d2);
        \draw (156d1)--(312d2)--(624d3);
        \draw (312d2)--(624d4);

        \draw (1d1)--(325d1)--(650d1)--(3250d1)--(9750d2)--(19500d7);
        \draw (650d1)--(6500d2);

        \draw (156d1)--(3900d1)--(15600d5);
        \draw (650d1)--(3900d1)--(15600d3);
        \draw (624d3)--(15600d5);

        \draw (1d1)--(300d1)--(600d2);

        \draw (1d1)--(624d8);
        \draw (1d1)--(9750d1);
        \draw (1d1)--(6500d1)--(13000d5);
    \end{tikzpicture}
    \caption{A lattice of common images of Galois}
    \label{fig:subgroup-lattice}
\end{figure}
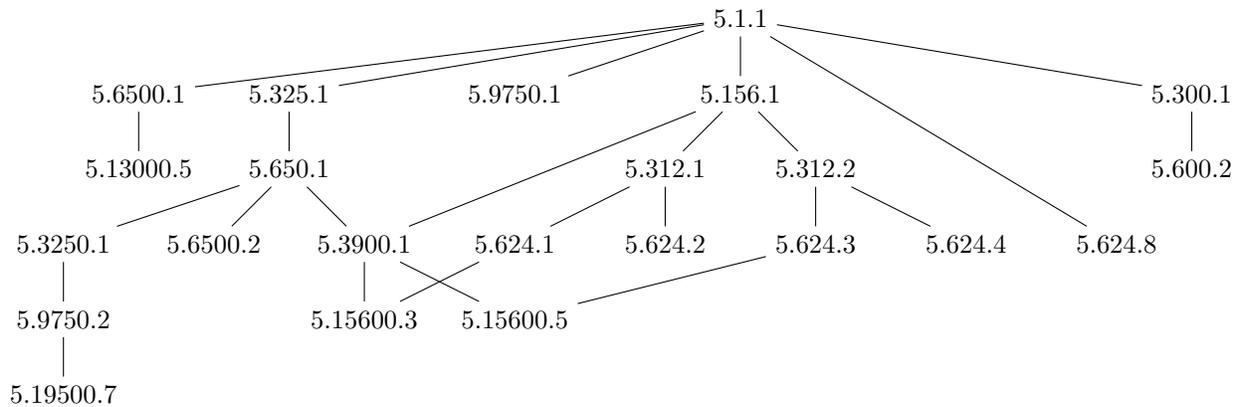

\begin{itemize}
    \item (5.156.1) This is a maximal subgroup of $\GSp_4(\FF_5)$, and specifically is the stabilizer of a 1-dimensional subspace. It so happens that 5.156.1 also stabilizes a 3-dimensional subspace containing the 1-dimensional one. Thus, elements of 5.156.1 are block upper triangular with block sizes of (1, 2, 1). In other words, with the right choice of basis, 
    \[
        5.156.1 = \left\{ \begin{bmatrix}
            *&*&*&*\\
            0&*&*&*\\
            0&*&*&*\\
            0&0&0&*\\
        \end{bmatrix} \right\} \cap \GSp_4(\FF_5).
    \]
    A Jacobian admits a rational 5-isogeny if and only if its image is contained in 5.156.1.
    \item (5.312.1) This is the index 2 subgroup of 5.156.1 obtained by restricting the top-left entry to be $\pm 1$. A Jacobian admits torsion defined over a quadratic field if and only if its image is contained in 5.312.1.
    \item (5.624.1) This is the index 4 subgroup of 5.156.1 obtained by restricting the upper left entry to be 1. A Jacobian admits rational 5-torsion if and only if $\im\rho_{J, 5} \leq 5.624.1$.
    \item (5.624.2) This is the index 2 subgroup of 5.312.1 characterized by $5.312.1 = 5.624.2 \times \{\pm 1\}$. The fact that 5.624.2 lacks the scalar matrix $-1$ implies that containment of an image in 5.624.2 may not be invariant under twisting - twists of Jacobians with image 5.624.2 may have image 5.312.1 \cite[p.~7]{minus1-in-subgroup}.
    \item (5.312.2) This subgroup is best understood by first understanding 5.624.3. This is the index 2 subgroup of 5.156.2 obtained by restricting the bottom-right entry to be $\pm 1$. There is a quotient $q$ onto the bottom right entry. If 5.312.2 is the image attached to $J$, then $q \circ \rho$, by virtue of being a continuous homomorphism $\Gal(\overline{\Q}/ \Q) \rightarrow \{\pm 1\}$, is a quadratic character $\chi_{\Q(\sqrt{d})}$. One may thus consider the restricted Galois representation $\rho_{J, 5}: \Gal(\overline{\Q}/\Q(\sqrt{d})) \rightarrow \GSp_4(\FF_5)$. The image of this restricted Galois representation is contained in 5.624.3. For the same reasons as discussed in the description of 5.624.3, it follows that such a Jacobian lacks quadratic 5-torsion but is isogenous to a variety with quadratic 5-torsion. 5.312.2 is the maximal possible image subject to this property.
    \item (5.624.3) This is the index 4 subgroup of 5.156.1 obtained by restricting the bottom right entry to be 1. Every element of this group has a non-trivial 1-eigenspace. Thus, Jacobians with image 5.624.3 have 5-torsion mod every prime of good reduction, but not over $\Q$. It follows by a theorem of Katz \cite[p.~483]{katz-local-global} that any such Jacobian is isogenous to a variety with rational 5-torsion. Conversely, any Jacobian with that property has image contained in 5.624.3.
    \item (5.624.4) This is the index 2 subgroup of 5.312.2 characterized by $5.312.2 = 5.624.4 \times \{\pm 1\}$. Like 5.624.2, this subgroup may contain the image attached to a Jacobian but not to all twists thereof.
    \item (5.325.1) This subgroup is maximal, and is the normalizer of 5.650.1. With the right choice of basis, 5.325.1 is the set of elements of either of the following forms:
    \[
        5.325.1 = \left\{\begin{bmatrix}
            *&*&0&0\\
            *&*&0&0\\
            0&0&*&*\\
            0&0&*&*\\
        \end{bmatrix}, \begin{bmatrix}
            0&0&*&*\\
            0&0&*&*\\
            *&*&0&0\\
            *&*&0&0\\
        \end{bmatrix}\right\} \cap \GSp_4(\FF_5).
    \]
    Jacobians admitting a quadratically-defined isogeny of degree prime to 5 to a product of non-isogenous elliptic curves have image contained in 5.325.1.
    \item (5.650.1) This is the group of elements preserving each summand is a direct sum decomposition of $J(\overline{\Q})[5]$ into two nondegenerate 2-dimensional subspaces each defined individually over $\FF_5$. If there is a rationally-defined isogeny $J \rightarrow E_1 \times E_2$ of degree relatively prime to 5 from a Jacobian $J$ to a product $E_1 \times E_2$ of non-isogenous elliptic curves, then $\im \rho_{J, 5} \leq 5.650.1$. The converse is notably false --- see Section \ref{sec: typicals}.
    \item (5.3250.1) This is the maximal subgroup of 5.650.1 whose action on one of the preserved 2-dimensional subspaces is the same as that of the exceptional maximal subgroup of $\GL_2(\FF_5)$. 
    \item (5.9750.2) This is the subgroup of 5.650.1 whose action on one of the preserved 2-dimensional subspaces is the same as that of the normalizer of a split Cartan subgroup of $\GL_2(\FF_5)$. 5.9750.2 is not maximal in 5.650.1; it is index 3 in the intermediary subgroup 5.3250.1. This corresponds to the fact that the normalizer of a split Cartan subgroup of $\GL_2(\FF_5)$ is not maximal, but instead contained in an exceptional maximal subgroup.
    \item (5.19500.7) This is an index 2 subgroup of 5.9750.2. Specifically, this is the subgroup of 5.650.1 whose action on one of the preserved 2-dimensional subspaces is the same as that of $G$, the intersection of the normalizer of a split Cartan and the normalizer of a non split Cartan subgroup of $\GL_2(\FF_5)$. Concretely, $G$ is the index 2 subgroup of the normalizer of a split Cartan for which the diagonal elements all have square determinant and the off-diagonal elements all have non-square determinant.
    \item (5.6500.2) This is the subgroup of 5.650.1 whose action on one of the preserved 2-dimensional subspaces is the same as that of the normalizer of a non-split Cartan subgroup of $\GL_2(\FF_5)$. 
    \item (5.3900.1) This subgroup is the intersection of 5.156.1 and 5.650.1. With the right choice of basis, one may express 5.3900.1 as 
    \[
        5.3900.1 = \left\{\begin{bmatrix}
            *&*&0&0\\
            0&*&0&0\\
            0&0&*&*\\
            0&0&*&*\\
        \end{bmatrix} \right\} \cap \GSp_4(\FF_5).
    \]
    \item (5.15600.3) This subgroup is the intersection of 5.624.1 and 5.650.1. It is the index-4 subgroup of 5.3900.1 for which the upper left entry is 1.
    \item (5.15600.5) This subgroup is the intersection of 5.624.3 and 5.650.1. It is the index-4 subgroup of 5.3900.1 for which the entry in the second row and second column is 1. 
    \item (5.624.8) This is an index 4 subgroup of the maximal subgroup 5.156.2, which is the stabilizer of a 2-dimensional isotropic subspace. With the right choice of basis, 5.624.8 is the subgroup consisting of all block upper triangular elements with blocks of size 2 such that the two diagonal blocks are equal. Alternatively, 5.624.8 is the intersection $Z_{\GL_4(\FF_5)}(\varphi) \cap \GSp_4(\FF_5)$ of $\GSp_4(\FF_5)$ with the centralizer in $\GL_4(\FF_5)$ of an element $\varphi$ with the following Jordan canonical form:
    \[
        \varphi \simeq \begin{bmatrix}
            3&1&0&0\\
            0&3&0&0\\
            0&0&3&1\\
            0&0&0&3
        \end{bmatrix}.
    \]
    The element $\varphi \in \GL_4(\FF_5)$ is so named because it shares a minimal polynomial with the golden ratio. 
    \item (5.300.1) This is a maximal subgroup. Precisely, this is the normalizer of 5.600.2.
    \item (5.600.2) This is an index 2 subgroup of 5.300.1. It is the subgroup preserving each summand in a direct sum decomposition of $\FF_5^4$ into two nondegenerate 2-dimensional subspaces defined jointly, but not individually, over $\FF_5$. Alternatively, 5.600.2 is the intersection $Z_{\GL_4(\FF_5)}(M) \cap \GSp_4(\FF_5)$ of $\GSp_4(\FF_5)$ with the centralizer in $\GL_4(\FF_5)$ of an element $M$ with the following Jordan canonical form:
    \[
        M \simeq \begin{bmatrix}
            0&1&0&0\\
            3&4&0&0\\
            0&0&0&1\\
            0&0&3&4
        \end{bmatrix}.
    \]
    \item (5.9750.1) This is a maximal subgroup of $\GSp_4(\FF_5)$. It is the normalizer of the subgroup preserving a direct sum decomposition of $\GSp_4(\FF_5)$ into two 2-dimensional isotropic subspaces, each defined individually over $\FF_5$. This group also coincides with the exceptional maximal subgroup $G_{1920}$ described in \cite[p.~7]{Vogt-surjectivity}.
    \item (5.6500.1) This is a maximal subgroup of $\GSp_4(\FF_5)$. It is the normalizer of 5.13000.5. 
    \item (5.13000.5) This is index 2 in the maximal subgroup 5.6500.1. It is the subgroup preserving a direct sum decomposition of $\GSp_4(\FF_5)$ into two 2-dimensional isotropic subspaces defined jointly, but not individually, over $\FF_5$.
\end{itemize}

\subsection{Endomorphism algebra $\Q$} \label{sec: typicals}

Of the 3990 Jacobians with non-surjective mod-5 image of Galois, 939 have endomorphism ring $\Z$. Among these, we determined a precise likely image for 898 curves. 

For each of the remaining $41$ curves, we were unable to distinguish between two or more possible images. For $37$ of the $41$ curves with undetermined images, this stems from an inability to definitively rule out the presence of simple quadratic torsion as noted in Remark \ref{rmk: false-negative}. However, even if Algorithm \ref{alg: finding_torsion} always returned \textrm{true} or \textrm{false}, the images of following four curves would still be undetermined.
\begin{itemize} 
    \item The curve given $y^2 = 4x^5+5x^4 -10x^3-25x^2+30x-3$ (37500.a.37500.1) is undetermined between subgroups 5.14976.1 and  5.14976.3. These are non-isomorphic subgroups which are both consistent with the presence of simple quadratic torsion.
    \item The curve given $y^2 = 4x^5-4x^4-16x^3-79x^2-76x-20$ (240250.a.240250.1) is undetermined between subgroups 5.14976.10 and 5.14976.11. These are non-isomorphic subgroups which are both consistent with a lack of simple quadratic torsion.
    \item The curve given $y^2 =-4x^6-12x^5-15x^4+10x^2+4x-3$ (400000.a.400000.1) is undetermined between subgroups 5.12480.19 and 5.12480.17. These are isomorphic as abstract groups and are both consistent with a lack of simple quadratic torsion.
    \item The curve given $y^2 = 4x^5+5x^4-10x^3-5x^2+10x+5$ (787500.a.787500.1) is undetermined between 5.14976.5 and 5.14976.7. These are non-isomorphic subgroups and are both consistent with a lack of quadratic torsion.
\end{itemize}

By far, the most common non-surjective image of Galois for Jacobians with endormophism ring $\Z$ is 5.624.1. Thus, the most common obstruction to surjectivity of the mod-5 image of Galois is presence of rational 5-torsion. The next most common obstruction is presence of quadratic 5-torsion. After these, the most common obstructions are the existence of isogenous varieties with the just-described obstructions. Perhaps surprisingly, it appears to be quite rare for Jacobians to admit 5-isogenies without an additional obstruction; only 16 Jacobians were computed to have mod-5 image of Galois 5.156.1. We expect that this behavior is a phenomenon only in small conductor, and that image 5.156.1 would be more common asymptotically.

Most Jacobians with image 5.650.1 are isogenous to a product of elliptic curves. The sole exception among all 3990 computed images is the curve $C: y^2 = x^6 + 8x^5 + 20x^4 + 20x^3 -8x -4$ (label \texttt{600000.b.600000.1} in the LMFDB). The Jacobian $J$ has geometric endomorphism ring $\Z$, and thus is geometrically simple. Despite this, $J$ has mod-5 image of Galois 5.650.1.

A summary of the results is available in Figure \ref{fig: typical-images-pie-chart}, with the full data in Table \ref{tab: typical image counts} in Appendix A.

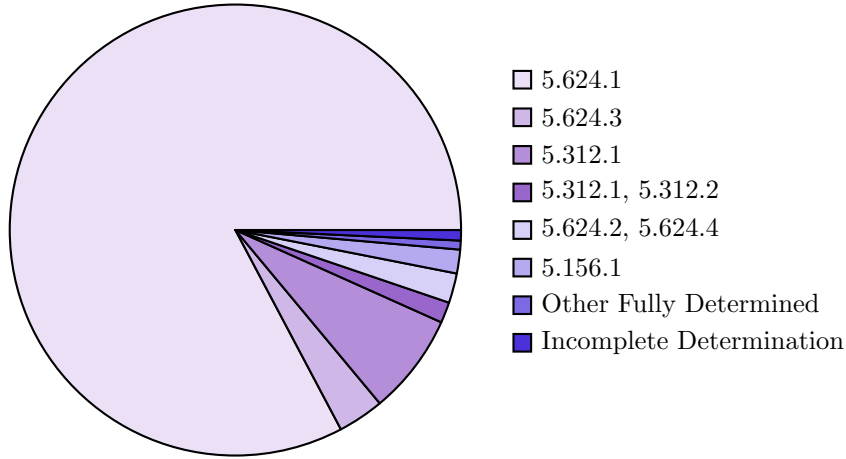
\begin{figure}[H]
    \centering
    {\Large \textbf{Images of Typical Curves}}\\
    \vspace{10pt}
    \begin{tikzpicture}
    \pie[sum = auto, hide number, color={amethyst!20, amethyst!47, amethyst!74, amethyst!100, indigo!20, indigo!37, indigo!65, indigo!90}, text = legend]{
    777/{5.624.1}, 
    31/{5.624.3}, 
    68/{5.312.1}, 
    14/{5.312.1, 5.312.2}, 
    20/{5.624.2, 5.624.4}, 
    16/{5.156.1}, 
    6/Other Fully Determined, 
    7/Incomplete Determination
}
    \end{tikzpicture}
    \caption{A breakdown of the results of Algorithm \ref{alg: imager} on typical curves.}
    \label{fig: typical-images-pie-chart}
\end{figure}

\subsection{Endomorphism algebra $\Q \times \Q$} \label{sec: Q_times_Q}

A majority (2468, to be precise) of the computed images were of Jacobians with endomorphism algebra $\Q \times \Q$. Such Jacobians are isogenous to products of non-isogenous elliptic curves lacking complex multiplication. We were able to compute the precise mod-5 image of Galois for all but twenty of the Jacobians in this category. Of those twenty, ten of those would be distinguishable if Algorithm \ref{alg: finding_torsion} never returned \texttt{maybe}, while the other 10 would still be undetermined. 

By far, the most common image in this category is 5.650.1, attached to just over 88\% of Jacobians with geometric endomorphism algebra $\Q \times \Q$. The next four most common images, in order, are 5.325.1, 5.15600.3, 15600.5, and 5.3900.1. These five images account for over 98\% of Jacobians in this category.

As an example, consider the curve $C: y^2 = x^6 - 4x^4 - 12x^3 - 16x^2 - 12x - 4$ (label \texttt{1573.a.1573.1} in the LMFDB) with Jacobian $J$. We compute that $\im \rho_{J, 5} = 5.15600.3$. From our description of this subgroup, we may expect a rational $n$-isogeny, for $n$ prime to 5, from $J$ to a product $E_1 \times E_2$ of non-isogenous elliptic curves for which $\rho_{E_1, 5}$ is surjective and $\rho_{E_2, 5}$ is the stabilizer of a vector $v \in \FF_5^2$. 

We indeed see exactly this; $J \simeq E_1 \times E_2$, for elliptic curves $E_1$ and $E_2$ given by $E_1: y^2+y=x^3-x^2-x-2$ and $E_2: y^2+y=x^3-x^2$ (LMFDB labels \texttt{143.a1} and \texttt{11.a3} respectively). An LMFDB lookup confirms that $\rho_{E_1, 5}$ and $\rho_{E_2, 5}$ are as we predict from our computations.

Most Jacobians with image 5.624.8 have a geometric endomorphism algebra isomorphic to a real quadratic field, but there are also seven Jacobians in the $\Q \times \Q$ category with image 5.624.8. The example with smallest conductor is the Jacobian of the curve $y^2 = 4x^6+4x^5+17x^4+26x^3+25x^2+36x+20$ (label \texttt{2028.a.64896.1} in the LMFDB). These curves admit isogenies of degree a multiple of 5 to products of elliptic curves.

\subsection{Endomorphism algebra $\Q \times \textrm{CM}$} \label{sec: Q_times_CM}

The third most common category of Jacobians is that consisting of Jacobians whose geometric endomorphism algebra is $\Q \times K$ for $K$ an imaginary quadratic field. Jacobians in this category are isogenous to a product of elliptic curves, exactly one of which admits complex multiplication. Our algorithm handles these Jacobians particularly well; the only curve in this category for which we are not able to determine the precise image has equation $C: y^2 = -3x^6 + 2x^5 - 5x^4 - 5x^2 + 2x - 3$ and label \texttt{16875.a.84375.1} in the LMFDB. The two candidates for $\im \rho_{\Jac(C), 5}$ are distinguished by whether they are consistent with presence of quadratic 5-torsion.

The vast majority of these Jacobians have image 5.9750.2 or image 5.6500.2. Of the remaining 8 Jacobians, 5 have image 5.19500.7. 

The failure of Jacobians in this category to have surjective image of Galois is often a well-understood consequence of how they split. For example, the curve in this category of minimal conductor is $C: y^2 = x^6 - 6x^4 + x^2 + 28$ (LMFDB label \texttt{448.a.448.2}). The Jacobian $J$ of $C$ splits via a rational isogeny into a product $E_1 \times E_2$ of elliptic curves given by $E_1: y^2+xy+y=x^3-x$ and $E_2: y^2=x^3-x$, according to the \cite{lmfdb}. Since $C$ is bielliptic, the isogeny $J \rightarrow E_1 \times E_2$ is of degree prime to 5. The first curve is not CM, and has maximal image of Galois $\rho_{E_1, 5}\cong \GL_2(\FF_5)$. The second curve, $E_2$, has potential complex multiplication; its geometric endomorphism ring is isomorphic to $\Z[i]$. The extra endomorphism $i$, given by $(x,y) \mapsto (-x, iy)$, is defined over $\Q(i)$. The endomorphism $i$ acts on $E_2[5]$ as $\left[\begin{smallmatrix}2&0\\0&3\end{smallmatrix}\right]$, so there is containment $\rho_{E_2, 5}(\Gal(\overline{\Q}/\Q(i))) \leq Z_{\GL_2(\FF_5)}(\left[\begin{smallmatrix}2&0\\0&3\end{smallmatrix}\right])$ of the image of the restricted Galois representation in the centralizer of $i$. As a linear transformation of $E_2[5]$, the extra endomorphism $i$ has centralizer $H$ a split Cartan subgroup of $\GL_2(\FF_5)$. The extension $\Q(i)/\Q$ is of course Galois, so $H$ is a normal subgroup of the image of Galois. Equivalently, the image of Galois is contained in the normalizer $N(H)$. Since the isogeny $J \rightarrow E_1 \times E_2$ has degree prime to 5, $$\im \rho_{J, 5} \leq (\im \rho_{E_1, 5} \times \im \rho_{E_2, 5}) \cap \GSp_4(\FF_5) \leq (\GL_2(\FF_5) \times N(H)) \cap \GSp_4(\FF_5) = 5.9750.2.$$ As expected, our computation of $\rho_{J, 5}$ yields $5.9750.2$. This confirms that there are no additional obstructions to surjectivity of $\rho_{J, 5}$ beyond those already described.

\subsection{Real quadratic endomorphism algebra} \label{sec: real_quad_alg}

There are 116 computed Jacobians with geometric endormorphism algebras isomorphic to real quadratic fields. Real quadratic fields lack 0-divisors, so Jacobians in this category are simple. 

Among these, the most common image is 5.624.8. That is plausible, for the following reason. Many Jacobians in the LMFDB with real multiplication happen to have endomorphism ring $\Z\left[\frac{\sqrt{5}+1}{2}\right]$. This may be a product of a phenomenon where, in small conductor, the endomorphism ring is of small discriminant. In any case, when $\End(J) \cong \Z\left[\frac{\sqrt{5}+1}{2}\right]$, the extra endomorphism generating $\Z\left[\frac{\sqrt{5}+1}{2}\right]$, which we denote $\varphi$, has minimal polynomial $x^2-x-1 \equiv (x-3)^2 \bmod 5$. Thus, $\varphi$ must have only 3 as an eigenvalue, and its Jordan canonical form must be comprised of Jordan blocks of sizes 2, 1, and 1, or of sizes 2 and 2. When $\varphi$ is defined over $\Q$, the image of Galois must be contained in the centralizer of $\varphi$. In the case where $\varphi$ consists of two Jordan blocks of size 2, the centralizer of $\varphi$ is exactly the subgroup 5.624.8. 

The next most common images are 5.600.2 and 5.300.1. These are plausible for similar reasons. For $d \equiv 2, 3 \bmod 5$, any order of $\Q(\sqrt{d})$ is generated by an element with minimal polynomial congruent to $x^2+x+2$ mod 5. Any element of $\GL_4(\FF_5)$ with minimal polynomial $x^2+x+2$ is conjugate to the matrix $M$ given in the description of 5.300.1. Suppose the rationally-defined endomorphism algebra of $J$ is $\Q(\sqrt{d})$ for $d$ non-square mod 5. In this case, there is an endomorphism which acts as $M$ on the 5-torsion. The subgroup 5.600.2 is the centralizer of $M$, so the image of Galois is contained in 5.600.2. If the extra endomorphism is defined over a quadratic extension, then we similarly obtain that the image of Galois is contained in 5.300.1.

Our algorithm has a slightly lower success rate with Jacobians of this type compared to the most common types; we are unable to identify a precise image for 12 of the 116 Jacobians in this category.

\subsection{Endomorphism algebra $M_{2\times 2}(\Q)$} \label{sec: M2Q_alg}

Jacobians with (geometric) endomorphism algebra $M_{2\times 2}(\Q)$ are isogenous to a product $E \times E'$ of isogenous elliptic curves. This cateogry of Jacobians gives our algorithm the most trouble; our algorithm computes a precise likely image for only 105 of the 147 Jacobians in this category.

The most common output of our algorithm for this category is 5.13000.5. The next most common output is $\{5.39000.2, 5.39000.7, 5.3900.8\}$, which accounts for 20 of the 42 curves in this category for which the precise likely image is unknown. This is unfortunate in that the three candidates feature large differences; for example they are contained in different maximal subgroups. Below is the full list of containments.
\begin{itemize}
    \item 5.39000.2 is contained in the maximal subgroups 5.156.2, 5.300.1, 5.325.1, and 5.9750.1. Note that 5.156.2 is not included in Figure \ref{fig:subgroup-lattice}; see the description of 5.624.8 for a description of 5.156.2.
    \item 5.39000.7 is contained in the maximal subgroups 5.300.1, 5.325.1, and 5.6500.1. 5.39000.7 is the intersection of (suitably chosen representatives of) 5.650.1 and 5.6500.1.
    \item 5.39000.8 is contained in the maximal subgroups 5.325.1 and 5.9750.1. 5.39000.8 is the intersection of (suitably chosen representatives of) 5.650.1 and 5.9750.1.
\end{itemize}

\subsection{Other endomorphism algebras} \label{sec: other_alg}

In the LMFDB, 17 genus 2 curves have Jacobians with geometric endomorphism algebras which are not of any of the so-far covered forms. Our algorithm determined a precise likely image of Galois for the Jacobians of 8 of these curves. Our algorithm does not employ knowledge of the endomorphism ring of the Jacobian. Because these Jacobians have comparatively large endomorphism rings, encorporating such information would likely yield an outsized improvement for curves in this category and curves with endomorphism algebra $M_{2 \times 2}(\Q)$.

\section*{Appendix A: Computational Results} 

Below are the results of our computations, broken up by endomorphism algebra type. Minimal-conductor examples are given as LMFDB labels.

\begin{longtable}{ | c | c | c | } 
  \caption{\centering Image counts for mod-$5$ Galois images of LMFDB curves with typical endomorphism ring} \\
  \hline
  \makebox[7cm][c]{Likely image(s)} & \makebox[2cm][c]{Count} & \makebox[5cm][c]{Minimal-conductor example} \label{tab: typical image counts}\\ 
  \hline
   $\{ 5.624.1 \}$	& 777 & 277.a.277.1 \\ 
     \hline
  $\{ 5.312.1 \}$ & 68 & 4293.a.4293.1 \\
  \hline
  $\{ 5.624.3 \}$ & 31 & 523.a.523.2 \\ 
    \hline
  $\{ 5.624.2, 5.624.4 \}$ & 20 & 6625.c.33125.1 \\
    \hline
  $\{ 5.156.1 \}$	& 16 & 8960.c.17920.1 \\
  \hline
  $\{ 5.312.2, 5.312.1 \}$ & 14 & 4672.a.9344.1 \\
    \hline
  $\{ 5.624.2 \}$ & 3 & 431250.a.431250.1 \\
    \hline
  $\{ 5.9360.13, 5.9360.15 \}$ & 2 & 100000.a.200000.1 \\
  \hline
  $\{ 5.78000.22 \}$ & 1 & 9125.a.228125.1 \\
  \hline
  $\{ 5.3744.1 \}$ & 1 & 18928.a.984256.1 \\
  \hline
  $\{ 5.14976.1, 5.14976.3\}$ & 1 & 37500.a.37500.1\\
  \hline
  $\{ 5.6240.2,  5.6240.4 \}$ & 1 & 84375.a.84375.1\\
  \hline
  $\{ 5.14976.10, 5.14976.11 \}$ & 1 & 240250.a.240250.1 \\
  \hline
  $\{ 5.12480.19, 5.12480.17 \}$	& 1 & 400000.a.400000.1\\
  \hline
  $\{ 5.650.1 \}$ & 1 & 600000.b.600000.1 \\
  \hline
  $\{ 5.14976.5, 5.14976.7 \}$ & 1 & 787500.a.787500.1 \\
  \hline
\end{longtable}

\begin{longtable}{ | c | c | c | } 
  \caption{Image counts for mod-$5$ Galois images of LMFDB curves with endomorphism algebra $\Q \times \Q$} \\
  \hline
  \makebox[7cm][c]{Likely image(s)} & \makebox[2cm][c]{Count} & \makebox[5cm][c]{Minimal-conductor example} \label{tab: QxQ image counts}\\ 
  \hline
   $\{ 5.650.1 \}$	& 2188 & 294.a.294.1 \\ 
  \hline
  $\{ 5.325.1 \}$ & 125 & 1088.a.1088.1 \\ 
  \hline
  $\{ 5.15600.3 \}$ & 68 & 363.a.11979.1 \\
  \hline
  $\{ 5.15600.5 \}$ & 30 & 847.d.847.1 \\
  \hline
  $\{ 5.3900.1\}$ & 19 & 8788.a.17576.1 \\
  \hline
  $\{5.14976.2, 5.14976.4\}$ & 10 & 726.a.1452.1 \\
  \hline
  $\{5.15600.4, 5.15600.6\}$ & 8 & 1125.a.151875.1 \\
  \hline
  $\{ 5.624.8 \}$ & 7 & 2028.a.64896.1 \\
  \hline
  $\{ 5.78000.22\}$ & 4 & 363.a.43923.1 \\
  \hline
  $\{ 5.3250.1\}$ & 3 & 52488.a.629856.1 \\
  \hline
  $\{ 5.234000.3 \}$ & 2 & 40000.a.160000.1 \\
  \hline
  $\{ 5.234000.6, 5.234000.4 \}$ & 2 & 40000.c.200000.1 \\
  \hline
  $\{ 5.9750.2\}$ & 2 & 520000.a.520000.1 \\
  \hline
\end{longtable}

\begin{longtable}{ | c | c | c | } 
  \caption{Image counts for mod-$5$ Galois images of LMFDB curves with endomorphism algebra $\Q \times K$ for $K$ an imaginary quadratic field.} \\
  \hline
  \makebox[7cm][c]{Likely image(s)} & \makebox[2cm][c]{Count} & \makebox[5cm][c]{Minimal-conductor example} \label{tab: CM image counts}\\ 
  \hline
   $\{  5.9750.2 \}$ & 163 & 448.a.448.2 \\ 
  \hline
  $\{ 5.6500.2\}$ & 132 & 686.a.686.1 \\ 
  \hline
  $\{ 5.19500.7 \}$ & 5 & 67500.a.810000.1 \\
  \hline
  $\{ 5.234000.3 \}$ & 1 & 1331.a.1331.1 \\
  \hline
  $\{ 5.468000.4, 5.468000.2\}$ & 1 & 16875.a.84375.1 \\
  \hline
  $\{5.468000.1\}$ & 1 & 16875.b.151875.1\\
  \hline
\end{longtable}

\begin{longtable}{ | c | c | c | } 
  \caption{Image Counts for mod-$5$ Galois Images of LMFDB Curves with real quadratic endomorphism algebras} \\
  \hline
  \makebox[7cm][c]{Likely image(s)} & \makebox[2cm][c]{Count} & \makebox[5cm][c]{Minimal-conductor example} \label{tab: RM image counts}\\ 
  \hline
   $\{ 5.624.8 \}$	& 55 & 529.a.529.1 \\ 
  \hline
  $\{ 5.600.2 \}$ & 29 & 841.a.841.1 \\ 
  \hline
  $\{ 5.300.1 \}$ & 10 & 36864.a.36864.1 \\
  \hline
  $\{ 5.74880.4 \}$ & 4 & 961.a.961.2 \\
  \hline
  $\{ 5.14976.2, 5.14976.4\}$ & 4 & 7569.a.68121.1 \\
  \hline
  $\{5.624.7\}$ & 4 & 180625.a.903125.1 \\
  \hline
  $\{5.14976.6, 5.14976.13, 5.14976.8, 5.14976.16 \}$ & 2 & 62500.a.1000000.1 \\
  \hline
  $\{  5.74880.16 \}$ & 1 & 961.a.961.1 \\
  \hline
  $\{ 5.74880.5\}$	& 1 & 961.a.923521.1 \\
  \hline
  $\{5.74880.13\}$ & 1 & 15625.a.15625.1 \\
  \hline
  $\{ 5.374400.2, 5.374400.10 \}$ & 1 & 12500.a.12500.1 \\
  \hline
  $\{ 5.374400.12, 5.374400.4, 5.374400.13 \}$ & 1 & 12500.b.50000.1 \\
  \hline
  $\{ 5.374400.5\}$ & 1 & 50000.b.800000.1 \\
  \hline
  $\{ 5.14976.1, 5.14976.3 \}$ & 1 & 112500.a.450000.1 \\
  \hline 
  $\{ 5.74880.8, 5.74880.21, 5.74880.14  \}$ & 1 & 378125.a.378125.1 \\
  \hline
\end{longtable}

\begin{longtable}{ | c | c | c | } 
  \caption{Image counts for mod-$5$ Galois images of LMFDB curves with endomorphism algebra $M_{2\times 2}(\Q)$} \\
  \hline
  \makebox[7cm][c]{Likely image(s)} & \makebox[2cm][c]{Count} & \makebox[5cm][c]{Minimal-conductor example} \label{tab: M2Q image counts}\\ 
  \hline
   $\{ 5.13000.5 \}$ & 51 & 169.a.169.1 \\ 
  \hline
  $\{ 5.39000.8, 5.39000.7, 5.39000.2\}$ & 20 & 69696.c.627264.1 \\ 
  \hline
  $\{ 5.6500.1 \}$ & 13 & 15552.c.746496.1 \\
  \hline
  $\{ 5.9750.1 \}$ & 12 & 4608.c.27648.1 \\
  \hline
  $\{ 5.78000.4, 5.78000.24\}$ & 8 & 196.a.21952.1 \\
  \hline
  $\{5.26000.1\}$ & 7 & 324.a.648.1 \\
  \hline
  $\{5.19500.3, 5.19500.2\}$ & 6 & 102400.b.102400.1 \\
  \hline
  $\{ 5.19500.1 \}$ & 4 & 6400.f.64000.1 \\
  \hline
  $\{ 5.39000.1\}$ & 4 & 25600.d.128000.1 \\
  \hline
  $\{ 5.13000.4 \}$ & 4 & 6075.a.18225.1 \\
  \hline
  $\{ 5.39000.17, 5.39000.6, 5.39000.5 \}$ & 2 & 3969.d.250047.1 \\
  \hline
  $\{ 5.374400.12, 5.374400.4, 5.374400.13 \}$ & 2 & 2500.a.50000.1 \\
  \hline
  $\{ 5.234000.24, 5.234000.23, 5.234000.18\}$ & 2 & 589824.a.589824.1 \\
  \hline
  $\{5.13000.8\}$ & 2 & 2187.a.6561.1 \\
  \hline
  $\{5.58500.2\}$ & 2 & 262144.b.524288.1 \\
  \hline
  $\{5.468000.25\}$ & 1 & 256.a.512.1\\
  \hline
  $\{5.936000.12, 5.936000.9, 5.936000.13\}$ & 1 & 576.a.576.1\\
  \hline
  $\{5.234000.34\}$ & 1 & 4096.e.524288.1\\
  \hline
  $\{5.19500.4\}$ & 1 & 8192.b.131072.1\\
  \hline
  $\{5.48750.1\}$ & 1 & 12544.g.175616.1\\
  \hline
  $\{5.19500.8\}$ & 1 & 12800.c.128000.1\\
  \hline
  $\{5.65000.5\}$ & 1 & 26244.d.314928.1\\
  \hline
  $\{5.13000.3\}$ & 1 & 26244.e.472392.1\\
  \hline
\end{longtable}

\begin{longtable}{ | c | c | c | } 
  \caption{Image counts for mod-$5$ Galois images of LMFDB curves with endomorphism algebra $M_{2\times 2}(K)$ for $K$ an imaginary quadratic field} \\
  \hline
  \makebox[7cm][c]{Likely image(s)} & \makebox[1cm][c]{Count} & \makebox[4cm][c]{Minimal-conductor example} \label{tab: M2(CM) image counts}\\ 
  \hline
   $\{ 5.390000.12, 5.390000.18, 5.390000.14, 5.390000.13, 5.390000.7\}$ & 4 & 4096.b.65536.1 \\ 
  \hline
  $\{ 5.585000.14, 5.585000.22\}$ & 1 & 5184.a.46656.1 \\ 
  \hline
  $\{ 5.1560000.7 \}$ & 1 & 40000.e.200000.1 \\
  \hline
  $\{  5.260000.11, 5.260000.16 \}$ & 1 & 2916.b.11664.1 \\
  \hline
  $\{ 5.130000.16\}$ & 1 & 11664.a.11664.1 \\
  \hline
\end{longtable}

\newpage

\begin{longtable}{ | c | c | c | } 
  \caption{Image counts for mod-$5$ Galois images of LMFDB curves with endomorphism algebra a quartic CM field} \\
  \hline
  \makebox[7cm][c]{Likely image(s)} & \makebox[2cm][c]{Count} & \makebox[5cm][c]{Minimal-conductor example} \label{tab: CM image counts}\\ 
  \hline
   $\{ 5.374400.5 \}$ & 2 & 10000.b.800000.1 \\ 
  \hline
  $\{ 5.1872000.4, 5.1872000.10, 5.1872000.1\}$ & 1 & 3125.a.3125.1 \\ 
  \hline
  $\{ 5.90000.1 \}$ & 1 & 28561.a.371293.1 \\
  \hline
  $\{ 5.187200.52 \}$ & 1 & 50625.a.253125.1 \\
  \hline
  $\{ 5.187200.52, 5.187200.53\}$ & 1 & 160000.c.800000.1 \\
  \hline
\end{longtable}

\begin{longtable}{ | c | c | c | } 
  \caption{Image counts for mod-$5$ Galois images of LMFDB Curves with non-split quaternion algebras as endomorphism algebras} \\
  \hline
  \makebox[7cm][c]{Likely image(s)} & \makebox[2cm][c]{Count} & \makebox[5cm][c]{Minimal-conductor example} \label{tab: QM image counts}\\ 
  \hline
   $\{ 5.19500.3, 5.19500.2 \}$ & 2 & 262144.d.524288.1 \\ 
  \hline
  $\{ 5.9750.1 \}$ & 1 & 20736.l.373248.1\\ 
  \hline
\end{longtable}

In addition to all the forms described above, it is possible for a Jacobian of a genus 2 curve to have endomorphism algebra $K \times K'$ for $K$ and $K'$ distinct imaginary quadratic fields. However, at the time of writing, there are no such curves in the LMFDB.

\section*{References}
\printbibliography[heading=none]

\end{document}